\documentclass{article}
\usepackage[utf8]{inputenc}
\usepackage{mathtools}
\usepackage{amsfonts}
\usepackage{amsmath}
\usepackage{amsthm}
\usepackage{amssymb}

\usepackage{graphicx}
\usepackage{hyperref}
\usepackage{arxiv}

\newcommand{\oN}{{\mathbb N}}
\newcommand{\oR}{{\mathbb R}}
\newcommand{\oZ}{{\mathbb Z}}

\newcommand{\ignore}[1]{}
\usepackage{pgf,tikz,pgfplots}
\usepackage{mathrsfs}
\usepackage{geometry}
\usetikzlibrary{arrows}

\newtheorem{theorem}{Theorem}[section]
\newtheorem{corollary}[theorem]{Corollary}
\newtheorem{lemma}[theorem]{Lemma}
\newtheorem{conjecture}{Conjecture}
\newtheorem{question}[theorem]{Question}
\newtheorem{proposition}[theorem]{Proposition}

\newtheorem{remark}[theorem]{Remark}
\newtheorem{example}[theorem]{Example}

\newcommand{\kr}{\mathcal{K}^{(r)}}
\newcommand{\Cr}{\mathcal{C}^{(r)}}
\newcommand{\C}{\mathcal{C}}
\newcommand{\quadr}{\mathcal{M}}
\newcommand{\fr}{f^{(r)}}
\newcommand{\lsim}{f}
\newcommand{\lsph}{F}
\usepackage{xcolor}
    \makeatletter
        \def\tagform@#1{\maketag@@@{$($#1$)$\@@italiccorr}}
    \makeatother
    \numberwithin{equation}{section}
    
\makeatletter
\renewcommand*\env@matrix[1][*\c@MaxMatrixCols c]{%
  \hskip -\arraycolsep
  \let\@ifnextchar\new@ifnextchar
  \array{#1}}
\makeatother

\newcommand{\supp}{\text{\rm Supp}}
\newcommand{\COP}{{\text{\rm COP}}}
\newcommand{\MK}{{\mathcal K}}
\newcommand{\MS}{{\mathcal S}}
\newcommand{\MC}{{\mathcal C}}
\newcommand{\MM}{{\mathcal M}}
\newcommand{\MT}{{\mathcal T}}
\newcommand{\rrank}{\vartheta\text{\rm -rank}}
\newcommand{\las}{\text{\rm las}}

\newcommand{\MoL}{\textcolor{black}}

\title{Finite convergence of sum-of-squares hierarchies for the stability number of a graph}

\author {Monique Laurent
\thanks{Centrum Wiskunde \& Informatica (CWI), Amsterdam, and Tilburg University. \url{monique.laurent@cwi.nl} }
\And 
Luis Felipe Vargas
\thanks{Centrum Wiskunde \& Informatica (CWI), Amsterdam. \url{luis.vargas@cwi.nl}
\newline
This work is supported by the European Union's Framework Programme for Research and Innovation Horizon
2020 under the Marie Skłodowska-Curie Actions Grant Agreement No. 813211  (POEMA).
} 
}
\date{Published electronically April 25, 2022}

\begin{document}

\maketitle

\begin{abstract}
We investigate a hierarchy of semidefinite bounds $\vartheta^{(r)}(G)$ for the stability number $\alpha(G)$ of a graph $G$, based on its copositive programming formulation and  introduced by 
 de Klerk and Pasechnik [{\em SIAM J. Optim.} 12 (2002), pp.875--892], who conjectured  convergence to $\alpha(G)$ in $r=\alpha(G)-1$ steps. Even the weaker conjecture claiming finite convergence is still open.  We establish links between this hierarchy and sum-of-squares hierarchies based on the Motzkin-Straus formulation of $\alpha(G)$, which we use to show finite convergence when $G$ is acritical, i.e., when $\alpha(G\setminus e)=\alpha(G)$ for all edges $e$ of $G$. 
This relies, in particular, on  understanding the structure of the  minimizers of the Motzkin-Straus formulation and showing that their number is finite precisely when $G$ is acritical. 
\MoL{Moreover we show that these results hold in the general setting of the weighted stable set problem for graphs equipped with positive node weights.} In addition, as a byproduct we show that deciding whether a standard quadratic program has finitely many minimizers does not admit a polynomial-time algorithm unless P=NP.
\end{abstract}



\section{Introduction}\label{sec-intro}

Given a graph $G=(V,E)$, its  {\em stability number} $\alpha(G)$   is defined as the largest cardinality of a stable set in $G$. Computing the stability number of a graph is a central problem in combinatorial optimization, well-known to be  NP-hard \cite{Karp}. \MoL{This problem (and the related one of finding a maximum cardinality clique in $G$, i.e., a maximum cardinality stable set in the complementary graph $\overline G$) has many applications in various areas, such as scheduling, social networks analysis, bioinformatics (see, e.g., \cite{Bomze-survey}, \cite{WuHao} and further references therein).}
Many approaches based, in particular,  on  semidefinite programming have been developed for constructing good relaxations. 
A starting point to define hierarchies of approximations for the stability number is the following formulation by Motzkin and Straus \cite{motzkin}, which expresses $\alpha(G)$ via quadratic optimization over the standard simplex $\Delta_n$:
\begin{align}\label{motzkin-form}\tag{M-S}
    \frac{1}{\alpha(G)}=\min \{x^T(A_G+I)x: x\in\Delta_n\}.
\end{align}
Throughout,  $\Delta_n=\{x\in \mathbb{R}^n \text{ : } x\geq0, \sum_{i=1}^{n}x_i=1\}$,  $A_G$ is the adjacency matrix of $G$, \MoL{$I$ is the identity matrix and $J$ is the all-ones matrix.}
Based on (\ref{motzkin-form}),  de Klerk and Pasechnik \cite{dKP2002} proposed the following reformulation:
\begin{align}
    \alpha(G)=  \min  \{t: x^T(t(I+A_G)-J)x \geq 0 \hspace{0.2cm} \text{ for all } x\in \mathbb{R}^n_+\},
\end{align}
which boils down to linear optimization over the copositive cone
$$
\COP_n:=\{M\in \mathcal S^n: x^TMx\ge 0 \ \forall x\in \oR^n_+\}.$$
Indeed, $\alpha(G)$  equals the smallest scalar $t$ for which  the matrix $M_{G,t}:= t(I+A_G)-J$ is copositive, i.e., belongs to $\COP_n$.
For $x\in \oR^n$ set $x^{\circ 2}:=(x_1^2,\ldots,x_n^2)$ and for a matrix $M\in\MS^n$ define the polynomials 
\begin{equation}\label{eqpPM}
p_M(x)=x^TMx \quad \text{ and } \quad P_M(x)=p_M(x^{\circ 2})=(x^{\circ2})^TMx^{\circ2}.
\end{equation}
Then   $M$  is copositive precisely when the polynomial $p_M$ is nonnegative over $\oR^n_+$ or, equivalently, when $P_M$ is nonnegative over $\oR^n$.  Based on this observation, 
 Parrilo \cite{Parrilo-thesis-2000} introduced the  following two subcones  of $\mathcal S^n$:
 \begin{align} \label{eqCKnr}
 \MC^{(r)}_n=\Big\{M: \Big(\sum_{i=1}^nx_i\Big)^r p_M(x)\in \mathbb{R}_+[x]\Big\}, \quad
 \MK^{(r)}_n=\Big\{M: \Big(\sum_{i=1}^nx_i^2\Big)^r P_M(x) \in\Sigma\Big\},
 \end{align}
which provide sufficient conditions for matrix copositivity:     $\C^{(r)}_n\subseteq \MK^{(r)}_n\subseteq \COP_n$  for any $r\ge 0$. 
Here $\oR_+[x]$ is the set of polynomials with nonnegative coefficients and $\Sigma$ denotes the set of sum-of-squares polynomials.
 De Klerk and Pasechnik \cite{dKP2002} used these two cones to define the following parameters: 
\begin{align} 
\zeta^{(r)}(G)= \min \{ t:  t(I+A_G)-J\in \MC^{(r)}_n\}, \quad\label{eqzetar}\\
\vartheta^{(r)}(G)= \min \{ t:  t(I+A_G)-J\in \MK^{(r)}_n\},\label{eqthetar}
\end{align}
which provide upper bounds on the stability number:   $\alpha(G)\le \vartheta^{(r)}(G) \le \zeta ^{(r)}(G)$. 
It is known that the program (\ref{eqzetar}) 
is feasible, i.e., $\zeta^{(r)}(G)<\infty$, if and only if $r\ge \alpha(G)-1$ and also that $\zeta^{(r)}(G)<\alpha(G)+1$, i.e., $\lfloor \zeta^{(r)}(G)\rfloor = \alpha(G)$, if and only if $r\ge \alpha(G)^2-1$ \cite{dKP2002,PVZ2007}.
On the other hand, the parameter $\vartheta^{(r)}(G)$ provides a nontrivial bound already at order $r=0$. Indeed, as shown in \cite{dKP2002},  the parameter  $\vartheta^{(0)}(G)$ coincides with $\vartheta'(G)$, the strengthening of the theta number $\vartheta(G)$ by Lov\'asz \cite{Lo79}, proposed in \cite{Sch79}. Recall that 
$$\vartheta(G)=\max\{\langle J,X\rangle: \text{Tr}(X)=1, X_{ij}=0 \ (\{i,j\}\in E), X\succeq 0\},$$
and $\vartheta'(G)$ is obtained by adding the nonnegativity constraint $X\ge 0$  to the above program. As is well-known we have 
\begin{equation}\label{eq-sandwich}
\alpha(G)\le \vartheta'(G)\le \vartheta(G)\le \chi(\overline G),
\end{equation}
 where $\chi(\overline G)$ denotes the coloring number of $\overline G$ (the complementary graph of $G$), i.e., the smallest number of cliques of $G$ needed to cover $V$.

Hence one can find $\alpha(G)$, {\em after rounding}, in $(\alpha(G))^2$ steps of the hierarchy $\zeta^{(r)}(G)$ or $\vartheta^{(r)}(G)$.
It is known that the linear bound $\zeta^{(r)}(G)$ is {\em never exact}: if $G$ is not the complete graph then $\zeta^{(r)}(G)>\alpha(G)$ for all $r$ \cite{PVZ2007}. On the other hand, de Klerk and Pasechnik \cite{dKP2002} conjecture that rounding is not necessary for the semidefinite parameter $\vartheta^{(r)}(G)$  and moreover that $\alpha(G)$ steps suffice to reach convergence.

\begin{conjecture}[De Klerk and Pasechnik \cite{dKP2002}]\label{conj1}
For any graph $G$ we have: $\vartheta^{(\alpha(G)-1)}(G)=\alpha(G)$.
\end{conjecture}
In fact, it is not even known whether finite convergence holds at some step, so also the following weaker conjecture is still open in general.
\begin{conjecture}\label{conj2}
For any graph $G$ we have:  $\vartheta^{(r)}(G)=\alpha(G)$ for some $r\in \mathbb{N}$.
\end{conjecture}
Let us call the smallest integer $r$ for which $\vartheta^{(r)}(G)=\alpha(G)$ the {\em $\vartheta$-rank} (or, simply, the {\em rank}) of $G$, denoted as $\rrank(G)$. \MoL{In other words, $\rrank(G)$ is the smallest integer $r$ for which the matrix $M_G:=\alpha(G)(I+A_G)-J$ belongs to the cone $\mathcal K^{(r)}_n$.} 
Then Conjecture \ref{conj2} asks whether the rank is finite for all graphs, while  Conjecture \ref{conj1} asks whether $\rrank(G)\le \alpha(G)-1$.

We recap some of the known results on these conjectures.
In view of (\ref{eq-sandwich}), if $\alpha(G)=\chi(\overline G)$ then $\vartheta^{(0)}(G)=\alpha(G)$ and thus $G$ has $\rrank$ 0; this holds, e.g., for perfect graphs \cite{Lo79}. Every graph satisfying $\vartheta(G)=\alpha(G)$ also has $\rrank$ 0; this is the case,  e.g, for the Petersen graph and, more generally, for Kneser graphs \cite{LovszKnesersCC}.
It is known that odd cycles, their complements  and odd wheels have $\vartheta$-rank 1 and thus satisfy Conjecture~\ref{conj1} \cite{dKP2002,PVZ2007}.
Conjecture~\ref{conj1} has been shown to hold for all graphs with $\alpha(G)\le 8$ in \cite{GL2007} (see also  \cite{PVZ2007} for the case $\alpha(G)\le 6$), but the general case is still wide open. 
Note that the conjectured bound $\alpha(G)-1$ on  $\rrank(G)$ is tight. As a first example,  the cycle $C_5$ has $\alpha(C_5)=2$ and $\rrank(C_5)=1$. As a second example, the complement of the icosahedron has $\alpha(G)=3$ and $\rrank(G)=2$; indeed,  $\rrank(G)\ge 2$ as $\vartheta^{(1)}(G)= 1+\sqrt 5>3$ \cite{dKP2002},  and $\rrank(G)\le 2$ as Conjecture~\ref{conj1} holds when $\alpha(G)=3$.

In this paper we want to further investigate the above conjectures.

\subsubsection*{Links to other hierarchies of Lasserre type}

Our approach is to relate the bounds $\vartheta^{(r)}(G)$ to other bounds that can be obtained by applying the  Lasserre hierarchy to the polynomial optimization problem (\ref{motzkin-form}). 
For this consider the polynomials $$f_G(x)=x^T(I+A_G)x \quad \text{ and } \quad F_G(x) =f_G(x^{\circ 2}) = (x^{\circ 2})^T (I+A_G)x^{\circ 2}.$$ That is, $f_G=p_{M}$ and $F_G=P_{M}$ for the matrix $M=I+A_G$ (recall (\ref{eqpPM})). 
Yet another reformulation of  (\ref{motzkin-form}) is that  $\alpha(G)$ can also be obtained via polynomial optimization over  the unit sphere: 
\begin{align}\label{motzkin-form-sq}\tag{M-S-Sphere}
{1\over \alpha(G)} =\min\Big\{F_G(x): x\in \oR^n, \sum_{i=1}^nx_i^2=1\Big\}.
\end{align} 
Now one can obtain bounds on $\alpha(G)$ by  applying the sum-of-squares approach of Lasserre \cite{Las2001}  to any of  the two formulations (\ref{motzkin-form}) and (\ref{motzkin-form-sq}). First we recall some notation. Given polynomials $g_0=1,g_1,\ldots,g_m\in \oR[x]$ and  $r\in\oN $ define the sets
\begin{align}
\MM(g_1,\ldots,g_m)_{r}=\Big\{\sum_{j=0}^m\sigma_j g_j: \sigma_j\in \Sigma, \deg(\sigma_jg_j)\le 2r\Big\},\label{eqQM}\\
\MT(g_1,\ldots,g_m)_r=\MM\Big(\prod_{j\in J} g_j: J\subseteq [m]\Big)_r,\label{eqPO}
\end{align}
known, respectively, as the quadratic module and the preordering generated by the $g_j$'s,   truncated at degree $2r$.
In addition, given polynomials $h_1,\ldots,h_k\in\oR[x]$, the set
\begin{equation}\label{eqideal}
\langle h_1,\ldots,h_k\rangle_r=\Big\{\sum_{i=1}^k u_i h_i: u_i\in\oR[x], \deg(u_ih_i)\le r\Big\}
\end{equation}
is the ideal generated by the $h_i$'s, truncated at degree $r$. 
Throughout, $\oR[x]_r$ denotes the set of polynomials with degree at most $r$ and  we set $\Sigma_r=\Sigma\cap\oR[x]_{2r}$, which consists of all \MoL{polynomials} of the form $\sum_i p_i^2$ for some $p_i\in \oR[x]_r$.
Corresponding to problems \eqref{motzkin-form} and \eqref{motzkin-form-sq} we now define  the parameters
\begin{align}
f^{(r)}_G= \sup\Big\{\lambda: f_G-\lambda \in \MM(x_1,\ldots,x_n)_{r} +
\Big\langle 1-\sum_{i=1}^n x_i\Big\rangle_{2r}\Big\}, \label{eqlassimplex}\\
f^{(r)}_{G,po}= \sup\Big\{\lambda : f_G-\lambda \in \MT(x_1,\ldots,x_n)_r +\Big\langle 1-\sum_{i=1}^nx_i\Big\rangle_{2r}\Big\}, \label{eqlaspo}
 \\
F^{(r)}_G=\sup\Big\{\lambda: F_G-\lambda \in \Sigma_{r}+\Big\langle 1-\sum_{i=1}^nx_i^2\Big\rangle_{2r} \Big\}, \label{eqlassphere}
\end{align}
which clearly satisfy $1/\alpha(G)\ge f^{(r)}_{G,po}\ge f^{(r)}_G$, 
$1/\alpha(G)\ge F^{(r)}_G$ and $F^{(2r)}_G\ge f^{(r)}_G$ for any $r\in\oN$.
We will  establish  further links, also  to  the parameters $\vartheta^{(r)}(G)$. 
In particular, we show that the approach based on  approximating the copositive cone by the cones $\mathcal K^{(2r)}_n$  (as in (\ref{eqthetar})) and the approach based on using the preordering truncated at degree $r+1$ (as in (\ref{eqlaspo})) are equivalent: for any $r\ge 0$ we have
\begin{align}\label{eqlink1}
   \frac{1}{\alpha(G)}\geq \dfrac{1}{\vartheta^{(2r)}(G)}=\lsph_G^{(2r+2)}= f^{(r+1)}_{G,po} \geq \lsim_G^{(r+1)}.
\end{align}
We say that {\em finite convergence} holds for the parameter $f^{(r)}_G$ if $f^{(r)}_G=1/\alpha(G)$ for some $r\in \oN$, and analogously for the other parameters. Based on the  inequalities  (\ref{eqlink1}) we see that finite convergence for the parameters $f^{(r)}_G$ implies finite convergence for the other parameters, and thus  in particular for $\vartheta^{(r)}(G)$, which would settle Conjecture~\ref{conj2}.

\subsubsection*{Role of critical edges}
Our first main result is  showing finite convergence of the bounds $f^{(r)}_G$  for the class of {\em acritical graphs.}
Recall that an edge $e$ of $G$ is said to be  {\em critical} if $\alpha(G\backslash e)=\alpha(G)+1$. The graph $G$ is called {\em $\alpha$-critical}  (or, simply, {\em critical}) when all its edges are critical, and  {\em acritical} when $G$  does not have any critical edge. For example, odd cycles are $\alpha$-critical  while even cycles are acritical. Critical edges and critical graphs have been studied in the literature; see, e.g.  \cite{LovPlumm}. It turns out that the notion of critical edges plays a central role in the study of the finite convergence of the above hierarchies of bounds.

On the one hand, it can be easily observed that deleting noncritical edges can only increase the $\rrank$. Indeed, if $\alpha(G\setminus{e})=\alpha(G)$ then $M_G-M_{G\setminus{e}}=\alpha(G)(A_G-A_{G\backslash e})$ is entry-wise nonnegative  
and thus belongs to $\MK^{(0)} \subseteq \MK^{(r)}$. Hence,  $M_{G\setminus{e}}\in \kr_n$ implies $M_G\in \kr_n$, which shows $\rrank(G)\le \rrank(G\setminus e)$.
Hence, after iteratively deleting noncritical edges, we obtain a subgraph $H$ of $G$ which is critical with $\alpha(H)=\alpha(G)$ and satisfies:
$\rrank(G)\le \rrank(H)$. As shown in Example \ref{excritical} below this inequality can be strict.
Therefore, finite convergence of the 
 parameters $\vartheta^{(r)}(G)$ (or $f^{(r)}_{G,po}, F^{(r)}_G)$ for the class of {\em critical} graphs implies the same property for general graphs.
Summarizing, it would suffice to show Conjectures  \ref{conj1} and  \ref{conj2} for  the class of {\em critical} graphs.

On the other hand, we can show finite convergence of the parameters $f^{(r)}_G$ for the class of {\em acritical} graphs 
and thus Conjecture \ref{conj2} holds  for acritical graphs (see Corollary \ref{finite-acritical-w}).

It turns out that critical edges  also play a crucial role in the analysis of the graphs with $\rrank$  0. In  the follow-up work \cite{LV2021} we  can indeed 
characterize the {\em critical} graphs with $\rrank$ 0, namely,   as those that \MoL{can} be covered by $\alpha(G)$ cliques, i.e., such that $\alpha(G)=\chi(\overline G)$. In addition,  in \cite{LV2021} we show that the problem of deciding whether a graph has $\rrank$ 0 can be algorithmically reduced to the same question restricted to the class of acritical graphs.

\begin{example}\label{excritical}
Consider the graph $G$ in Figure \ref{fig-GH1H2}, obtained by adding one \MoL{pendant} node to the cycle $C_5$.
Then,  $\alpha(G)=3=\overline\chi( G)$ and thus  $\rrank(G)=0$. 
Note that $G$ has two critical subgraphs $H_1$ and $H_2$ with $\alpha(H_1)=\alpha(H_2)=3$,  shown in Figure \ref{fig-GH1H2}: $H_1$ is $C_5$ with an isolated node, which has   $\rrank(H_1)=1$  (see, e.g., \cite{dKP2002}), while $H_2$ consists of three independent edges with  $\rrank(H_2)=0$ (since $\alpha(H_2)=\overline \chi(H_2)=3$). 
\end{example}

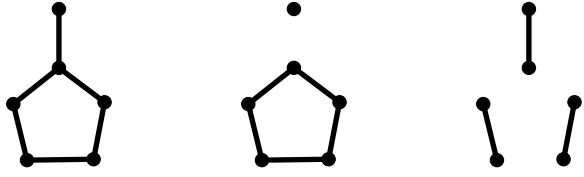
\begin{figure}\centering
	\definecolor{ududff}{rgb}{0.30196078431372547,0.30196078431372547,1.}
	\definecolor{f}{rgb}{0., 0., 0.}
	\begin{tikzpicture}[line cap=round,line join=round,>=triangle 45,x=.6cm,y=.6cm]
	\clip(-1.44,1.34) rectangle (3.42,6.28);
	\draw [line width=2.pt] (-0.02,3.18)-- (0.98,3.98);
	\draw [line width=2.pt] (0.98,3.98)-- (1.98,3.22);
	\draw [line width=2.pt] (1.74,1.96)-- (1.98,3.22);
	\draw [line width=2.pt] (0.28,1.94)-- (-0.02,3.18);
	\draw [line width=2.pt] (0.28,1.94)-- (1.74,1.96);
	\draw [line width=2.pt] (0.98,3.98)-- (0.98,5.28);
	\begin{scriptsize}
	\draw [fill=f] (0.98,3.98) circle (2.5pt);
	\draw [fill=f] (-0.02,3.18) circle (2.5pt);
	\draw [fill=f] (0.28,1.94) circle (2.5pt);
	\draw [fill=f] (1.74,1.96) circle (2.5pt);
	\draw [fill=f] (1.98,3.22) circle (2.5pt);
	\draw [fill=f] (0.98,5.28) circle (2.5pt);
	\end{scriptsize}
	\end{tikzpicture}
	\definecolor{ududff}{rgb}{0.30196078431372547,0.30196078431372547,1.}
	\begin{tikzpicture}[line cap=round,line join=round,>=triangle 45,x=.6cm,y=.6cm]
	\clip(-1.44,1.34) rectangle (3.42,6.28);
	\draw [line width=2.pt] (1.74,1.96)-- (1.98,3.22);
	\draw [line width=2.pt] (0.28,1.94)-- (-0.02,3.18);
	\draw [line width=2.pt] (-0.02,3.18)-- (0.98,3.98);
	\draw [line width=2.pt] (0.98,3.98)-- (1.98,3.22);
	\draw [line width=2.pt] (1.74,1.96)-- (0.28,1.94);
	\begin{scriptsize}
	\draw [fill=f] (0.98,3.98) circle (2.5pt);
	\draw [fill=f] (-0.02,3.18) circle (2.5pt);
	\draw [fill=f] (0.28,1.94) circle (2.5pt);
	\draw [fill=f] (1.74,1.96) circle (2.5pt);
	\draw [fill=f] (1.98,3.22) circle (2.5pt);
	\draw [fill=f] (0.98,5.28) circle (2.5pt);
	\end{scriptsize}
	\end{tikzpicture}
	\definecolor{ududff}{rgb}{0.30196078431372547,0.30196078431372547,1.}
	\begin{tikzpicture}[line cap=round,line join=round,>=triangle 45,x=.6cm,y=.6cm]
	\clip(-1.44,1.34) rectangle (3.42,6.28);
	\draw [line width=2.pt] (1.74,1.96)-- (1.98,3.22);
	\draw [line width=2.pt] (0.28,1.94)-- (-0.02,3.18);
	\draw [line width=2.pt] (0.98,3.98)-- (0.98,5.28);
	\begin{scriptsize}
	\draw [fill=f] (0.98,3.98) circle (2.5pt);
	\draw [fill=f] (-0.02,3.18) circle (2.5pt);
	\draw [fill=f] (0.28,1.94) circle (2.5pt);
	\draw [fill=f] (1.74,1.96) circle (2.5pt);
	\draw [fill=f] (1.98,3.22) circle (2.5pt);
	\draw [fill=f] (0.98,5.28) circle (2.5pt);
	\end{scriptsize}
	\end{tikzpicture}
	\caption{Graph $G$ (left), graph $H_1$ (middle), graph $H_2$ (right)}\label{fig-GH1H2}
\end{figure}

\subsubsection*{Number of global minimizers and finite convergence}
A main  reason why critical edges play a role in the study of finite convergence comes from the fact that problem (\ref{motzkin-form}) has infinitely many global minimizers when $G$ has critical edges. Indeed, next to the global minimizers arising from the maximum stable sets (of the form  $\chi^S/\alpha(G)$ with $S$ stable of size $\alpha(G)$), also some special convex combinations of them are global minimizers  when $G$ has  critical edges (see Corollary \ref{cor_global_min}). \MoL{Note that the existence of spurious mininizers (i.e., not directly arising from maximum stable sets) is well-known, see, e.g., \cite{Bomze1997,PJ1996}.}
Our approach to prove finite convergence of the bounds $f^{(r)}_G$ is to apply a result by Nie \cite{Nie} (itself based on the so-called Boundary Hessian Condition of Marshall~\cite{Marshall2006}), which requires to check whether the classical sufficient optimality conditions hold at all global minimizers of \eqref{motzkin-form}. These conditions imply in particular that the problem must have finitely many minimizers, which explains why we can only apply it to acritical graphs.


There is a \MoL{well-known} easy remedy to force having finitely many minimizers, simply by perturbing the Motzkin-Straus formulation \eqref{motzkin-form}. Indeed, if we replace  the adjacency matrix $A_G$ by $(1+\epsilon)A_G$ for any $\epsilon>0$, then the corresponding quadratic program still has optimal value $1/\alpha(G)$, but now the only global minimizers are those arising from the maximum stable sets. To get this property it would in fact suffice to perturb the adjacency matrix at the positions corresponding to the critical edges of $G$.
For the hierarchies of parameters obtained via this perturbed formulation we can show the finite convergence property, \MoL{see Theorem \ref{theo-finite-w} (which applies to the general setting of weighted graphs as discussed below).}
 However, since we do not know a  bound on the order of  convergence, which {\em does not depend on $\epsilon$},  it remains unclear how this can be used to derive the finite convergence of the original (unperturbed) parameters. 

Nevertheless, as a byproduct of our analysis of the minimizers of the (perturbed) Motzkin-Straus formulation, we can show NP-hardness of the problem of deciding whether a standard quadratic optimization problem has finitely many global minimizers. The key idea is to reduce it to the problem of testing critical edges which is itself NP-hard, see Section \ref{sec-complexity}.

\MoL{\subsubsection*{Extension to the weighted stable set problem}
Our results extend to the general setting of weighted graphs $(G,w)$, where $w\in \oR^V$ is a positive node weight vector, i.e., with $w_i>0$ for all $i\in V$.
Then $\alpha(G,w)$ denotes the maximum weight $w(S)=\sum_{i\in S}w_i$ of a stable set $S$ in $G$, with $\alpha(G,e)=\alpha(G)$ for the all-ones weight vector $w=e=(1,\ldots,1)$.
The following analogue of Motzkin-Straus formulation has been shown in \cite{GHPR}:
\begin{equation}\label{motzkin-form-w}\tag{M-S-weighted}
{1\over \alpha(G,w)}=\min \{p_B(x)=x^TBx: x\in \Delta_n\},
\end{equation}
where the matrix $B$ is of the form $B=B_w+A$, with $(B_w)_{ii}=1/w_i$, $A_{ii}=0$ ($i\in V$), 
$(B_w)_{ij}=(1/w_i+1/w_j)/2$, $A_{ij}\ge 0$ ($\{i,j\}\in E$), and $(B_w)_{ij}=A_{ij}=0$ ($\{i,j\}\not\in E$).
In  the case $w=e$ we have $B_e=I+A_G$; hence, if we select $A=0$ then we find the original Motzkin-Straus program (\ref{motzkin-form}) and if we select $A=\epsilon A_G$ then we find the perturbed Motzkin-Straus formlation mentioned in the previous paragraph.
There is a natural weighted analogue of critical edges: call an edge $\{i,j\}$  {\em $w$-critical} in $G$ if there exists $R\subseteq V$ such that both $R\cup\{i\}$ and $R\cup \{j\}$ are stable sets with $\alpha(G,w)=w(R\cup\{i\})=w(R\cup \{j\})$. 
Then, program (\ref{motzkin-form-w}) has finitely many minimizers if and only if $A_{ij}>0$ for all edges $\{i,j\}\in E$ that are $w$-critical and, in that case, the sufficient optimality conditions hold at all minimizers (see Proposition \ref{prop-finiteMSw}).
In addition, in that case, we can show the finite convergence of the semidefinite 
bounds $\vartheta^{(r)}(G,w)$ (the weighted analogues of $\vartheta^{(r)}(G)$) to $\alpha(G,w)$ when $G$ has no $w$-critical edge (see Section~\ref{sec-convergence-w}). 
 }

\subsubsection*{Links to related literature}
Given a graph $G$ define  the poynomial
$Q_G(x) =(x^{\circ 2})^T(\alpha(G)(I+A_G)-J)x^{\circ 2}$, which  is an even form (i.e., a homogeneous polynomial, with all variables appearing with an even degree) with degree 4. As $Q_G$  is nonnegative on $\oR^n$,  by Artin's theorem, it can be written as a sum of squares of rational functions: $Q_G= \sum_{i=1}^m p_i^2/h^2$ for some $p_i,h\in\oR[x]$. Then, what Conjecture \ref{conj2} claims is that the denominator $h^2$ can be chosen to be of the form $(\sum_i x_i^2)^r$ for some $r\in \oN$. Note that if $Q_G$ would be strictly positive (i.e., vanish only at the origin) then this claim would follow from a result of P\'olya \cite{Polya} (see also Reznick \cite{Reznick1995}). However, the polynomial $Q_G$ is not strictly positive, since any global minimizer of problem (\ref{motzkin-form}) provides  a nonzero root of $Q_G$ lying in $\Delta_n$. \MoL{On the positive side, in \cite{CPR} it is shown that   P\'olya's result still holds for some nonnegative even forms $Q$ with zeros (assuming, among others, that they are located at the corners of the simplex).}
In addition, Scheiderer \cite{Scheiderer2006} shows that if $Q$ is an arbitrary  form in three variables that is nonnegative on $\oR^3$ then it is indeed true that $(\sum_{i=1}^3x_i^2)^r Q\in \Sigma $  for some $r\in \oN$. 
On the negative side, for any $n\ge 4$, there are examples  of $n$-variate nonnegative polynomials $Q$ for which  $(\sum_i x_i^2)^rQ\not\in\Sigma$ for all $r\in \oN$;  such $Q$ can be chosen to be an even form of degree 4 for \MoL{$n\ge 6$ (see \cite{LV2021})}.
So Conjecture~\ref{conj2} claims a rather remarkable property for the class of forms $Q_G$ (and Conjecture \ref{conj1} claims an even stronger property).
 In this paper we will show that Conjecture \ref{conj2} holds when the graph $G$ is acritical, which corresponds to the case when  the polynomial $Q_G$ has  finitely many zeros in the simplex $\Delta_n$. We will in fact show this property for a larger class of degree 4 even forms (\MoL{see Section \ref{sec-convergence-w}}).
 
Our approach  relies on considering the Lasserre  hierarchy (\ref{eqlassimplex}) for problem (\ref{motzkin-form}) and  using the fact that its finite convergence implies finite convergence of the hierarchy $\vartheta^{(r)}(G)$  (in view of (\ref{eqlink1})). 
The goal is thus to  show  finite convergence of Lasserre hierarchy (\ref{eqlassimplex}) or, equivalently, that 
the polynomial $f_G-1/\alpha(G) =x^T(I+A_G)x-1/\alpha(G)$ belongs to the quadratic module $\mathcal M(x_1,\ldots,x_n, \pm (1-\sum_i x_i)).$
The question of identifying sufficient conditions for finite convergence of Lasserre hierarchy applied to a  polynomial optimization problem has been much studied in the literature; see,
 in particular, the works by Scheiderer \cite{Scheiderer2005,Scheiderer2006}, Marshall \cite{Marshall2006,Marshall2008,Marshall2009}, Kriel and Schweighofer \cite{KS1,KS2}, Nie \cite{Nie}, and references therein.
  Assume $f$ is a polynomial nonnegative on a basic closed semialgebraic set $K$ defined by polynomial (in)equalities, whose associated quadratic module $\mathcal M$ is Archimedean. Marshall \cite[Theorem 1.3]{Marshall2009} gives a set of algebraic conditions on the zeros of the polynomial $f$ in the  set $K$, known as the {\em Boundary Hessian Condition} (BHC), that ensures that $f$ belongs to the quadratic module $\mathcal M$.  Nie \cite{Nie} shows that (BHC) holds if the natural sufficient optimality conditions hold at all the global minimizers of $f$ over $K$ and thus Lasserre hierarchy has  finite convergence  (see  Theorem \ref{theo-Nie} below). 
 \MoL{ This result is also used in the recent work \cite{BM} to investigate  finite convergence (and exactness) of the dual moment hierarchy. }
  Note that a restriction to the application of these results is that  these optimality conditions (and  (BHC)) can  hold only when the number of global minimizers is finite. Since these conditions depend on the optimization problem, one faces the same issues also when using the (richer) preordering instead of the quadratic module. We make this remark in view of the equivalent reformulation of the parameters $\vartheta^{(r)}(G)$ in terms of the preordering-based  hierarchy $f^{(r)}_{G,po}$ mentioned in (\ref{eqlink1}).
Let us also mention that  while the result in \cite{Marshall2009} does not come with a degree bound for the order of the relaxation where finite convergence takes place, such a degree bound is given in  \cite{KS1}. However the results in \cite{KS1} require (among others)  the \MoL{additional} restriction  that the finitely many global minimizers should all lie in the interior of the set $K$, which is not the case for problem (\ref{motzkin-form}), neither  for  its perturbations  introduced in the paper. Finally, there are other results that show finite convergence of the Lasserre hierarchy, for instance, under some convexity assumptions (see \cite{dKL2011,Lasserre2009}), or when  the semi-algebraic set $K$ is finite (see Nie \cite{Nie2013}), or when the description of the set $K$ is enriched with various additional polynomial constraints (e.g., arising from  KKT conditions) (see, e.g.,  \cite{HNW,Nie2019} and further references therein).

 There is also interest in the literature in understanding when the first level of Lasserre hierarchy (also known as the Shor relaxation or the basic semidefinite relaxation) is exact when applied to quadratic optimization problems (see, e.g., the recent papers  \cite{BY2020,WK2021} and further 
 references therein). For standard quadratic programs, where one wants to minimize a quadratic form $p_M(x)=x^TMx$ over $\Delta_n$, we characterize the set of matrices $M$ for which the first level relaxation is exact. Moreover, we show that this holds precisely when the first level relaxation\ is feasible (see Lemma~\ref{ch-qp-level-1}). In the special  case of problem (\ref{motzkin-form}), when $M=I+A_G$, the first level relaxation gives the parameter $f^{(1)}_G$, which will be shown to  be exact (i.e., equal to $1/\alpha(G)$) precisely when the graph $G$ is a disjoint union of cliques (see Lemma \ref{lemexactfG1}).
 One can also use the preordering instead of the quadratic module and  ask when the corresponding first level relaxation is exact. For problem (\ref{motzkin-form}) this amounts to asking when $f^{(1)}_{G,po}=1/\alpha(G)$ or, equivalently (in view of (\ref{eqlink1})), when $\vartheta^{(0)}(G)=\alpha(G)$. Characterizing these graphs seems difficult in general, but, when restricting to critical graphs,  $\vartheta^{(0)}(G)=\alpha(G)$ if and only if $G$ can be covered by $\alpha(G)$ cliques  (see \cite{LV2021}). \MoL{This question is considered in   \cite{GY2021}  for a general standard quadratic program, asking to find the minimum value $p_{\min}$ of  $p_M(x)=x^TMx$ over  $\Delta_n$.  
An algebraic characterization is given there of the matrices $M$ for which  equality $p_{\min}=\Theta^{(0)}_M$ holds, where $\Theta^{(0)}_M$ is the analogue of $\vartheta^{(0)}(\cdot)$ as defined  in (\ref{relax-copositive}), together with some concrete classes of matrices achieving this equality.}
 
Finally, let us point out that the hierarchies considered in this paper are all based on continuous formulations of the stability number.  Alternatively, one can formulate $\alpha(G)$ as the maximum value of $\sum_{i\in V}x_i$ taken over all  $x$ in the discrete  cube $\{0,1\}^n$ that satisfy the edge constraints $x_i+x_j\le 1$ for all $\{i,j\}\in E$. One can model  the binary variables by the quadratic constraints $x_i^2=x_i$ ($i\in [n]$) and apply 
 the Lasserre/Parrilo approach, which provides a hierarchy of bounds, known to converge to $\alpha(G)$ in finitely many steps, in fact in $\alpha(G)$ steps  \cite{Lasserre2001b,Laurent2003}. When adding suitable nonnegativity conditions one gets the parameters $\las_r(G)$ that satisfy $\alpha(G)\le \las_r(G)\le \vartheta^{(r-1)}(G)$ for any $r\ge 1$ and $\las_1(G)=\vartheta^{(0)}(G)$   \cite{GL2007}.  Hence, what  Conjecture \ref{conj1} claims is that the continuous copositive-based hierarchy $\vartheta^{(r)}(G)$ has the same finite convergence behaviour as the discrete formulation-based Lasserre hierarchy. As observed above this question is also relevant to several other interesting aspects of real algebraic geometry.
 
\subsubsection*{Organization of the paper}
The paper is organized as follows. In Section \ref{sec-Preliminaries}, we recall the classical  optimality conditions in nonlinear programming and their use to show finite convergence  of the Lasserre hierarchy for polynomial optimization. In Section \ref{section-link}  we link several sum-of-squares approximation hierarchies for standard quadratic programs and we discuss some questions about the feasibily and exactness of these relaxations and their application to the Motzkin-Straus formulation (\ref{motzkin-form}).
Section \ref{minimizers-w} is focused on the study of the minimizers of  problem (\ref{motzkin-form}), where, in particular, we prove that (\ref{motzkin-form}) has finitely many minimizers precisely when the graph is acritical. \MoL{We  in fact prove these results in the general setting of weighted graphs $(G,w)$ and show that the number of global minimizers of the program (\ref{motzkin-form-w}) is finite precisely when there are no $w$-critical edges in $G$.}  In Section \ref{sec-convergence-w} we apply  the previous results  to show finite convergence of the semidefinite hierarchy $\vartheta^{(r)}(G)$ when $G$ is  acritical \MoL{or, more generally, of the hierarchy $\vartheta^{(r)}(G,w)$ to $\alpha(G,w)$ when $G$ has no $w$-critical edge.} In addition, we consider perturbed hierarchies for the (weighted) stability number and we give some facts and open questions about them. In Section \ref{sec-complexity}, we investigate the complexity of the problem of deciding whether a standard quadratic program has finitely many minimizers. 

\subsubsection*{Notation}

Notation about polynomials will be given in Section \ref{sec-Preliminaries}, but here 
we group  some notation about graphs and matrices  that is used throughout the paper. 
Given a graph $G=(V=[n],E)$, a set $S\subseteq V$ is stable if it does not contain an edge, and $\alpha(G)$ is the maximum cardinality of a stable set. A set $C\subseteq V$ is a clique if any two distinct vertices in $C$ are adjacent, and  $\chi(\overline G)$ denotes the minimum number of cliques whose union is $V$.  
For a set $S\subseteq V$ and a vertex $j\in V\setminus S$, we 
let $N_S(j)=\{i\in S: \{i,j\}\in E\}$ denote the set of vertices  $i\in S$ that are adjacent to $j$. 
An edge $e\in E$ is critical if $\alpha(G\setminus e)=\alpha(G)+1$, $G$ is called critical if all its edges are critical and $G$ is called acritical if none of its edges is critical.
Observe that $G$ is acritical precisely when 
$|N_S(j)|\ge 2$ for every stable set $S$ with $|S|=\alpha(G)$ and every $j\in V\setminus S$. For a subset $U\subseteq V$, $G[U]$ denotes the induced subgraph, with vertex set $U$ and edges the pairs $\{i,j\}\in E$ that are contained in $U$. 
For a vector $x\in \oR^n$  we let $\supp(x)=\{i\in [n]: x_i \ne 0\}$ denote the support of $x$. In addition, $e=(1,\ldots,1)^T$ denotes the all-ones vector, $\{e_1,\ldots,e_n\}$ denotes the standard unit basis of $\oR^n$, $I\in \mathcal S^n$ denotes the identity matrix  and  $J=ee^T\in \mathcal S^n$  the all-ones matrix. We also use the symbols $J_n$ and $J_{n,m}$ to denote the all-ones matrix of size $n\times n$ and $n\times m$, respectively.

\ignore{
\subsection{Something we dont want to forget}
\begin{theorem}\cite{scheiderer}
Let $p$ be a positive semideifnite form in 3 variables, then there exists $N=N(p)$ so that $(x^2+y^2+z^2)^Np(x,y,z)$ is a sum of squares. Indeed, $x^2+y^2+z^2$ can be replaced by any positive definite form.
\end{theorem}
Denote by $P_{n,m}$ the psd forms of degree $n$ in $m$ variables.
Reznick proved in \cite{Reznick_absence} that there is no single form $h$ such that $ph$ is SOS if $p\in P_{n,m}$ is positive semidefinite. Moreover, he show that there is not even a finite set of forms $H$ such that if $p\in P_{n,m}$ then $ph$ is SOS for some $h\in H$. In particular, the exponent in the previous theorem is unbounded.

\medskip
{\bf Another thing not to forget: } There are papers studying conditions ensuring that  the first level of Lasserre (aka the Shor relaxation) of quadratic programs is exact. Look at them and see what this means for our problem; this should be exactness of the bound $f^{(0)}_G$. See in particulier these papers:\\
- Burer, Ye: \verb1http://www.optimization-online.org/DB_HTML/2018/02/6461.html
\\
- Wang and Kilinc-Karzan: \verb1https://arxiv.org/abs/1911.09195
\\
- Locatelli: \verb1http://www.optimization-online.org/DB_HTML/2018/02/6461.html
\\
and references there.
}

\section{Preliminaries on polynomial optimization}\label{sec-Preliminaries}

Given polynomials $f$, $g_j$ for $j\in [m]$, and $h_i$  for $i \in [k]$,
consider the polynomial optimization problem:
\begin{align}\label{poly-opt}\tag{P}
f_{\min}=\inf\{f(x): \MoL{g_j(x)}\ge 0\ (j\in [m]), h_i(x)=0 \ (i \in [k])\}=\inf\{f(x): x\in K\},
\end{align}
setting $K=\{x\in\oR^n: g_j(x)\ge 0\ (j\in [m]), h_i(x)=0\ (i\in [k])\}$.
A well-known approach  for solving problem \eqref{poly-opt} is the Lasserre-Parrilo  approach,  which is based on using positivity certificates arising from suitable sums of squares representations for polynomials that are nonnegative over the feasible set $K$.  Such positivity certificates arise by considering the (truncated) quadratic module,  preordering and  ideal introduced in relations \eqref{eqQM}, \eqref{eqPO} and \eqref{eqideal}.
Set $g=(g_1,\ldots,g_m)$ and $h=(h_1,\ldots,h_k)$ for a short-hand, and    $\quadr(g)=\bigcup_{r\ge 0}\quadr(g)_r$,
$\langle h\rangle =\bigcup_{r\ge 0}\langle h\rangle _r$. Then $\MM(g)+\langle h\rangle $ is said to be {\em Archimedean} if the polynomial $R^2-\sum_{i=1}^{n}x_i^2$ belongs to $\mathcal{M}(g)+\langle h\rangle$ for some $R\in \mathbb{R}$. Note this implies that $K$ is compact. The following results by  Schm\"udgen \cite{Schmudgen} and Putinar \cite{Putinar} play a central role in polynomial optimization.

\begin{theorem} \label{theoPuSch}
Assume the feasible region $K$ of \eqref{poly-opt} is compact. Then any polynomial that is strictly positive on $K$ belongs to $\MT(g)+\langle h\rangle$ (Schm\"udgen \cite{Schmudgen}). If in addition $\MM(g)+\langle h\rangle$ is Archimedean, then any polynomial that is strictly positive on $K$ belongs to $\MM(g)+\langle h\rangle$ (Putinar \cite{Putinar}).
\end{theorem}
\noindent
Using  the truncated quadratic module and preordering leads to the  parameters:
\begin{align}
   \fr:= \sup\{ \lambda:  f-\lambda \in \quadr(g)_r+\langle h \rangle_{2r}\},\label{eqfr}\\
   \fr_{po}:= \sup\{ \lambda :  f-\lambda \in \mathcal{T}(g)_r+\langle h \rangle_{2r}\},\label{eqfrpo}
\end{align}
to which we will  refer as the {\em Lasserre hierarchy} (or the {\em sum-of-squares hierarchy}), sometimes adding the adjective `preordering-based' when referring to $\fr_{po}$. Clearly we have  $\fr \le \fr_{po}\leq f_{\min}$, $\fr\le f^{(r+1)}$  and $\fr_{po} \leq f^{(r+1)}_{po}$ for all $r$. 
As a direct application of Theorem \ref{theoPuSch}, the parameters $\fr_{po}$ converge asymptotically to $f_{\min}$ when $K$ is compact, while the (possibly weaker) parameters $\fr$ also converge asymptotically to $f_{\min}$ under the Archimedean condition.
 We are interested in problems for which the Lasserre hierarchy has {\em finite} convergence. We say the  parameters $\fr$  have {\em finite convergence} if $\fr=f_{\min}$ for some $r\in \mathbb{N}$;  analogously for the parameters $\fr_{po}$. 
 
\medskip In order to prove finite convergence of the Lasserre hierarchy for some special classes of polynomial optimization problems, we will use a result of Nie \cite{Nie}, which relies on the optimality conditions for nonlinear optimization. So we start with a quick recap on these optimality conditions, which we state here for problem \eqref{poly-opt} though they hold in a more general setting (see, e.g.,  \cite{Bertsekas/99}).

Let $u$ be a local minimizer of problem (\ref{poly-opt}) and let $J(u)=\{j\in[m]\text{ : } g_j(u)=0\}$ be the index set of the active inequality constraints at $u$.
We say that the {\textit{constraint qualification condition (CQC)}} holds at $u$ if the gradients of the active constraints at $u$ are linearly independent:
\begin{equation}\label{eqCQC}\tag{CQC}
\text{The vectors in }\{\nabla g_j(u): j\in J(u)\}\cup \{ \nabla h_i(u): i\in [k]\} \text{ are linearly independent.}
\end{equation}
If \eqref{eqCQC} holds at $u$ then there exist Lagrange multipliers $\lambda_1, \dots, \lambda_k, \mu_1, \dots, \mu_m\in\oR$  satisfying
\begin{align}\label{eqFOOC} 
  \nabla f(u)=\sum_{i=1}^{k}\lambda_i\nabla h_i(u)+\sum_{j=1}^{m}\mu_j\nabla g_j(u),\tag{FOOC}\\
 \label{eqCC}   \mu_1g_1(u)=0,\dots, \mu_m g_m(u)=0, \mu_1\geq 0, \dots \mu_m\geq 0.\tag{CC}
\end{align}
The condition (FOOC) is known as  the {\em first order optimality condition} and (CC) as the {\em complementarity condition}. If  it holds that 
\begin{align}\label{eqSCC}\tag{SCC}
    \mu_j> 0 \text{ for every } j\in J(u), \quad \mu_j=0 \text{ for } j\in [m]\setminus J(u),
\end{align}
then we say that the {\textit{strict complementarity condition }}(SCC) holds at $u$.
Define the Lagrangian function 
$$L(x)= f(x)-\sum_{i=1}^{k}\lambda_ih_i(x)-\sum_{j\in J(u)}\mu_jg_j(x).$$
Another necessary condition for $u$ to be a local minimizer is the {\textit{second order necessity condition}} (SONC):  
\begin{align}\label{eqSONC}\tag{SONC}
    v^T\nabla^2L(u)v\geq 0 \text{ for all } v\in G(u)^{\perp},
\end{align}
where $G(u)$ is the matrix with rows the gradients of the active constraints at $u$ and $G(u)^\perp$ is its kernel:
$$G(u)^{\perp}=\{x\in \oR^n: x^T\nabla g_j(u)=0 \text{ for all  } j\in J(u) \text{ and }  x^T\nabla h_i(u)=0 \text{ for all } i\in [k]\}.$$
 If  it holds that 
\begin{align}\label{eqSOSC}\tag{SOSC}
    v^T\nabla^2L(u)v> 0 \text{ for all } 0\neq v\in G(u)^{\perp},
\end{align}
then we say that the \textit{second order sufficiency condition}  (SOSC) holds at $u$. The relations between these optimality conditions and the local minimizers are summarized in the following classical result.

\begin{theorem}[see, e.g.,  \cite{Bertsekas/99}]\label{theo-local}
    Let $u$ be a feasible solution of problem \eqref{poly-opt}.
\begin{description} 
    \item[(i)] Assume $u$ is a local minimizer of \eqref{poly-opt} and (\ref{eqCQC}) holds at $u$. Then the conditions \eqref{eqFOOC}, \eqref{eqCC} and \eqref{eqSONC} hold at $u$.
    \item[(ii)]  Assume that   (\ref{eqFOOC}), (\ref{eqSCC}) and (\ref{eqSOSC}) hold at $u$. Then $u$ is a strict local minimizer of \eqref{poly-opt}.
\end{description} 
\end{theorem}
The relation between the optimality conditions for problem \eqref{poly-opt} and finite convergence of the parameters $\fr$ is given by the following result of Nie \cite{Nie}. 

\begin{theorem}[Nie \cite{Nie}]\label{theo-Nie}
Consider problem \eqref{poly-opt} and the parameters $f^{(r)}$ from (\ref{eqfr}).
Assume  that  the Archimedean condition holds, i.e., $R^2-\sum_{i=1}^nx_i^2\in \MM(g)+\langle h\rangle$ for some $R\in\oR$,  and that the constraint qualification \eqref{eqCQC}, strict complementary \eqref{eqSCC} and second order sufficency \eqref{eqSOSC} conditions hold at every global minimizer of \eqref{poly-opt}. Then  Lasserre's hierarchy $\fr$ has finite convergence, i.e., we have $f^{(r)}=f_{\min}$  for some $r\in \mathbb{N}$.
\end{theorem}
Note that, under the assumptions of Theorem \ref{theo-Nie}, all global minimizers of (\ref{poly-opt}) are {\em strict} minimizers (by Theorem \ref{theo-local} (ii)) and thus (\ref{poly-opt}) has {\em finitely many} global minimizers.
\MoL{(For if not, there exists a  sequence $(x_i)_i\subseteq K$, where all $x_i$ are global minimizers of $f$ over $K$. Under the Archimedean condition $K$ is compact and thus this sequence has an accumulation point $x^*\in K$. Then $x^*$  is also a global minimizer, but it is not a strict minimizer, yielding a contradiction.)}

\section{Links between the various  hierarchies}\label{section-link}

In this section we prove relation (\ref{eqlink1}), which establishes links between the various hierarchies of bounds $\vartheta^{(r)}(G)$, $f^{(r)}_G$ $f^{(r)}_{G,po}$ and $F^{(r)}_G$ from relations (\ref{eqthetar}), (\ref{eqlassimplex}), (\ref{eqlaspo}) and (\ref{eqlassphere}).  We start with establishing these links   in the more general setting   of standard quadratic programs. 

\subsection{Links between the hierarchies for standard quadratic programs}\label{secSQP}

Given a symmetric matrix $M\in \MS^n$, recall the polynomials  $p_M(x)=x^TMx$ and $P_M(x)=p_M(x^{\circ 2})$ from \eqref{eqpPM}. We consider the 
following {\em standard quadratic optimization problem}:
\begin{align}\label{quadr-simplex}
p_{\min} =\min\Big\{p_M(x): x\in \Delta_n\Big\},
\end{align} 
which can  be equivalently reformulated as optimization over the unit sphere:
\begin{align}\label{quadr-sphere}
    p_{\min}=\min \Big\{ P_M(x): x\in \oR^n, \sum_{i=1}^{n}x_i^2=1\Big\}.
\end{align}
In analogy to definitions  (\ref{eqlassimplex}), (\ref{eqlaspo}) and (\ref{eqlassphere}) we can define  the corresponding sum-of-squares hierarchies for both problems (\ref{quadr-simplex}) and (\ref{quadr-sphere}), and the preordering-based  hierarchy for the simplex formulation (\ref{quadr-simplex}), leading to the parameters
\begin{align}
    p_M^{(r)}=\max \Big\{\lambda: p_M - \lambda \in \quadr(x_1, x_2, \dots, x_n)_r +  \Big\langle\sum_{i=1}^{n} x_i-1\Big\rangle_{2r}\Big\},\label{las-simplex-general} \\
    p_{M,po}^{(r)}=\max\Big\{\lambda: p_M-\lambda\in \mathcal{T}(x_1,x_2,\dots, x_n)_r + \Big \langle \sum_{i=1}^{n}x_i-1\Big\rangle_{2r}\Big\},\label{las-preo-general}\\
    P_M^{(r)}=\max\Big\{\lambda: P_M-\lambda \in \Sigma_r +\Big\langle\sum_{i=1}^{n}x_i^2-1\Big\rangle_{2r}\Big\}
    \label{las-sphere-general}
    \end{align}
for any integer $r\ge 1$. 
Observe that  we are in the Archimedean setting and that the above programs are feasible for any $r\ge 2$. To see  this  one can use  the following identities: 
for any $i\in [n]$,
$$1-x_i= 1-\sum_{k=1}^n x_k +\sum_{k\in[n]\setminus \{i\}}x_k, \quad
1-x_i^2 = {(1+x_i)^2\over 2}(1-x_i) + {(1-x_i)^2\over 2}(1+x_i).
$$
This implies $n-\sum_ix_i^2 \in \mathcal{M}(x_1,\ldots,x_n)_2 +\langle 1-\sum_i x_i\rangle_4$, thus showing the Archimedean condition holds. We next verify feasibility of the programs.  If $M\succeq 0$ then the polynomial $p_M$ belongs to $ \Sigma_1$ and thus the programs defining $p_M^{(1)}, p^{(1)}_{M,po}, P^{(2)}_M$ are feasible. Otherwise,  $\mu:=\lambda_{\min}(M)<0$ and $p_M(x)-n\mu= x^T(M-\mu I)x -\mu(n-\sum_i x_i^2)$, which shows feasibility of the programs defining $p_M^{(r)}, p_{M,po}^{(r)}, P_M^{(r)}$ for $r\ge 2$. In addition note that $p^{(1)}_{M,po}$ is  finite when  $M$ is entry-wise nonnegative.
Observe also that   the optimum is attained in the above programs since  the search region for $p-\lambda$ is a closed set (see \cite{Marshall2003}).

Now, we characterize the set of matrices $M$ for which the program (\ref{las-simplex-general}) is feasible at order $r=1$. Moreover, we prove that in that case the program is exact, i.e.,  $p_M^{(1)}=p_{\min}$.

\begin{lemma}\label{ch-qp-level-1}
Given a symmetric matrix $M\in \mathcal S^n$, the following assertions are equivalent.
	\begin{description}
	\item[(i)] 	The program (\ref{las-simplex-general}) is feasible for $r=1$, i.e., $p_M^{(1)}$ is finite.
	\item[(ii)]  There exist $\lambda \in \mathbb{R}$ and $a\in \mathbb{R}^n_{+}$ such that 
		$M-\lambda J-(ae^T+ea^T)/2 \succeq 0.$
		\item[(iii)] $p^{(1)}_M=p_{\min}$.
		\end{description}
	\end{lemma}
	
	\begin{proof}
		We first prove (i) $\Longleftrightarrow$ (ii). Assume program (\ref{las-simplex-general}) is feasible, i.e.,   there exist $\lambda\in \mathbb{R},$ $a\in \mathbb{R}^n_{+}$, $Q\succeq 0$ and $u(x)\in \mathbb{R}[x]$ such that  
		$$ x^TMx-\lambda=x^TQx + a^Tx + (e^Tx-1)u(x).$$
		Then there exists $v(x)\in \mathbb{R}[x]$ such that   
		$$ x^TMx-\lambda (e^Tx)^2=x^TQx + (a^Tx)(e^Tx) + (e^Tx-1)v(x).$$
\MoL{		Indeed, we can select $v(x)=u(x)-\lambda(1+e^Tx)$, which follows from
$$x^TMx-\lambda (e^Tx)^2= x^TMx -\lambda +\lambda(1-(e^Tx)^2) = x^TQx+a^Tx +(e^Tx-1)(u(x)-\lambda(1+e^Tx)).$$}
		Hence the quadratic polynomial  $x^T(M-\lambda J -Q-(ae^T+ea^T)/2)x$  vanishes on $\{x: e^Tx=1\}$ and thus on $\oR^n$, which implies $M-\lambda J -Q-(ae^T+ea^T)/2=0$ and thus (ii) holds. The argument can be clearly reversed, which shows the equivalence of (i) and (ii). 
	
As (iii) implies (i) it 		suffices now to show  (ii) $\Longrightarrow$ (iii).
By the above argument, if (ii) holds then we have
\begin{align}\label{programp}
		p_M^{(1)}= \sup\{ \lambda: \lambda\in \oR, a\in \oR^n_+, \ M-\lambda J-(ae^T+ea^T)/2\succeq  0 \}.
		\end{align}		
Define the matrices $A_i=(e_ie^T+ee_i^T)/2$ for $i\in [n]$. Then the  dual program of (\ref{programp})  reads
\begin{equation}\label{programd}
\inf\{\langle M,X\rangle: \langle J,X\rangle =1, \langle A_i,X\rangle \ge 0 \ (i\in [n]), X\succeq 0\}.
\end{equation}
As program (\ref{programd}) is strictly feasible and bounded from below by $p_M^{(1)}$, strong duality holds and the optimum value of (\ref{programd}) is equal to $p^{(1)}_M$. 
We now show that $p_{\min}\le p^{(1)}_M$.
For this  let $X$ be feasible for (\ref{programd}) and define the vector $x=Xe$. Then $x\in \Delta_n$ since $x_i=\langle A_i,X\rangle \ge 0$ for all $i\in [n]$,  and $e^Tx=\langle J,X\rangle=1$, which implies 
$x^TMx\ge p_{\min}$. 
In addition,   we have $X-xx^T\succeq 0$, which follows from the fact that
$$\left(\begin{matrix}1 & x^T\cr x & X\end{matrix}\right)\succeq 0,$$ 
(as $X\succeq 0$, $x=Xe$ and $e^TXe=1$).
Consider also a feasible solution $(\lambda, a)$ to  (\ref{programp}), so that
$M-\lambda J-\sum_{i=1}^n a_i A_i \succeq 0$. Then we have $\langle M-\lambda J-\sum_i a_iA_i, X-xx^T\rangle \ge 0$ which, combined with $\langle J, X-xx^T\rangle =0$ and $\langle A_i, X-xx^T\rangle =0$ for all $i\in [n]$, implies that 
$\langle M,X\rangle \ge x^TMx\ge p_{\min}$ and thus $p^{(1)}_M \ge p_{\min}$, as desired.
	\end{proof}

Here is an immediate consequence of the reformulation of the parameter $p^{(1)}_M$ given in (\ref{programp}), that we will need later.
\begin{lemma}\label{lem-finite-pM1}
Assume that the program (\ref{programp}) defining $p^{(1)}_M$ is feasible, i.e., $M=\lambda J+Q+(ae^T+ea^T)/2$ for some $\lambda \in \oR$, $Q\succeq 0$ and $a\in \oR^n_+$.
Then, for any $i\ne j\in [n]$, we have
$M_{ii}+M_{jj}-2M_{ij}=Q_{ii}+Q_{jj}-2Q_{ij}\ge 0$.
In addition, if $M_{ii}+M_{jj}-2M_{ij}=0$ then $Q(e_i-e_j)=0$.
\end{lemma}
\begin{proof} Direct verification.
\end{proof}

Alternatively, following \cite{BDKRQ,dKP2002},  problem  (\ref{quadr-simplex}) can be reformulated as a copositive program:
\begin{align}\label{quadr-copositive}
    p_{\min} =\max\Big\{\lambda : M-\lambda J\in \COP_n\Big\}.
\end{align} 
By replacing the cone $\COP_n$ by its subcone $\MK^{(r)}_n$ we now obtain the following lower bound for $p_{\min}$:
\begin{align}\label{relax-copositive}
   \Theta_M^{(r)} :=\max\Big\{\lambda : M-\lambda J \in \MK^{(r)}_n\Big\}
\end{align} 
for any integer $r\ge 0$. Note that  $\lambda=\min_{i,j}M_{ij}$ provides a feasible solution for (\ref{relax-copositive}) since then $M-\lambda J$ belongs to $\mathcal K_n^{(0)}$. We begin with the following easy relationships among the above  parameters.

\begin{lemma}\label{sim-sph}
    For all $r\ge 1$  we have: $\max\{p^{(r)}_M, p^{(r)}_{M,po},P^{(r)}_M, \Theta^{(r)}_M\}  \le p_{\min}$ and $p_M^{(r)}\leq \min\{ P_M^{(2r)}, p_{M,po}^{(r)}\}$.
\end{lemma}
\begin{proof}
That all parameters are lower bounds for $p_{\min}$ follows from their definition, the inequality $p^{(r)}_M\le p^{(r)}_{M,po}$ follows from the inclusion $\MM(x_1,\ldots,x_n)_r\subseteq \MT(x_1,\ldots,x_n)_r$ and, for the inequality $p^{(r)}_M\le P^{(2r)}_M$, note that $p_M-\lambda\in\MM(x_1,\ldots,x_n)_r$ $+\langle 1-\sum_ix_i\rangle_{2r}$ implies 
$P_M-\lambda\in \Sigma_{2r}+\langle 1-\sum_ix_i^2\rangle_{4r}$.
\end{proof}


Following  \cite{dKLP} we can now relate  the bounds in (\ref{las-sphere-general}) and (\ref{relax-copositive}). For this we use the   following result from \cite{dKLP} (see Proposition 2 and Lemma 1 there).

\begin{theorem}[de Klerk et al. \cite{dKLP}]\label{theo1}
Let $q$ be a form of even degree $2d\geq 2$. For any $r\in \oN$  we have:
$$ q(x)\left (\sum_{i=1}^{n}x_i^2\right)^r \in \Sigma_{r+d} \Longleftrightarrow q\in \Sigma_{r+d}+ \Big\langle1-\sum_{i=1}^{n} x_i^2\Big\rangle_{2(r+d)}. $$
\end{theorem}

\begin{lemma}\label{equiv-sos-lass}
 For any $M\in\MS^n$  and $r\ge 0$, we have: $\Theta_M^{(r)}=P_M^{(r+2)}$.
\end{lemma}

\begin{proof}
By definition, $\Theta^{(r)}_M$ is the largest $\lambda$ for which the matrix $M-\lambda J$ belongs to the cone $ \MK^{(r)}_n$ or, equivalently, the polynomial $
(\sum_ix_i^2)^r (P_M(x)-\lambda(\sum_i x_i^2)^2)$ belongs to $\Sigma_{r+2}$. In view of Theorem~\ref{theo1} this is equivalent to requiring that $P_M-\lambda(\sum_i x_i^2)^2$ belongs to $\Sigma_{r+2}+\langle 1-\sum_ix_i^2\rangle_{2(r+2)}$. Now observe that  
\begin{equation}\label{eq0}
P_M(x)-\lambda=P_M(x) -\lambda\Big(\sum_ix_i^2\Big)^2 + \lambda\Big(\Big(\sum_ix_i^2\Big)^2-1\Big),
\end{equation}
 where $(\sum_ix_i^2)^2-1=(\sum_ix_i^2-1)(\sum_ix_i^2+1)\in \langle 1-\sum_ix_i^2\rangle_{2(r+2)}$. From this we obtain the desired identity $\Theta^{(r)}_M=P_M^{(r+2)}$.
\end{proof}


Next we relate the preordering-based bound $p^{(r)}_{M,po}$ (for the simplex formulation) and the Lasserre bound $P^{(r)}_M$ (for the sphere formulation). For this we use the following result of  \cite{ZVP2006}.

\begin{theorem}[Pena et al. \cite{ZVP2006}]\label{theo-squares-vera}
Let $q$ be a homogeneous polynomial of degree $d$ and define the polynomial $Q(x)=q(x^{\circ 2})$. Then, $Q\in \Sigma_d$  if and only if $q$ can be decomposed as
\begin{align}
    q(x)=\sum_{\substack{R\subseteq[n] \\
    |R|\leq d, |R|\equiv d \ (\text{\rm mod } 2)}}\sigma_R(x)\prod_{i\in R}x_i,
\end{align}
where $\sigma_R$ is a homogeneous polynomial with  degree at most $d-|R|$ and $\sigma_R\in \Sigma$.
\end{theorem}

As an application we recall the  characterization for the cone $\MK^{(0)}_n$, consisting  of the matrices $M$ for which the polynomial $P_M(x)=(x^{\circ 2})^TMx^{\circ 2}$ is a sum of squares.

\begin{proposition} [Parrilo \cite{Parrilo-thesis-2000}]\label{prop-K0}
A matrix $M$ belongs to $\MK^{(0)}_n$ if and only if there exist matrices $P\succeq 0$ and $N\ge 0$ such that $M=P+N$, where we may assume without loss of generality that $N_{ii}=0$ for all $i\in [n]$.
\end{proposition}

\begin{lemma}\label{link-preo_simp}
For any $M\in\MS^n$ and $r\ge 0$, we have: $p_{M,po}^{(r)}=P_M^{(2r)}$.
\end{lemma}

\begin{proof}
First, assume  $\lambda$ is feasible for $p^{(r)}_{M,po}$, i.e., 
$$p_M(x)-\lambda=\sum_{\substack{R\subseteq[n] \\
    |R|\leq r, |R|\equiv r \ (\text{\rm mod} 2)}}\sigma_R(x)\prod_{i\in R}x_i + u(x)\left(1-\sum_{i=1}^{n}x_i\right),$$
where $\sigma_R\in\Sigma$ is a form of degree at most $2r-|R|$  and $\deg(u)\leq 2r-1$. Replacing throughout $x$ by $x^{\circ 2}$ we obtain a decomposition of $P_M-\lambda$ in $\Sigma_{2r}+\langle 1-\sum_ix_i^2\rangle_{4r}$, which shows that 
$P_M^{(2r)}\ge p_{M,po}^{(r)}$. We now show the reverse inequality.
For this assume  $\lambda$ is feasible for $P_M^{(2r)}$, i.e.,
$$P_M(x)-\lambda=\sigma(x)+\left(1-\sum x_i^2\right)u(x),$$
where $\sigma\in \Sigma_{2r}$  and $\deg(u)\leq 4r-2$.  Hence, using  (\ref{eq0}),  the homogeneous polynomial 
$P_M(x)-\lambda(\sum_ix_i^2-1)^2$ belongs to $\Sigma_{2r}+\langle 1-\sum_ix_i^2\rangle_{4r}$.
Applying Theorem \ref{theo1} to it we can conclude that 
$$\Big(\sum_{i=1}^{n}x_i^2\Big)^{2r-2}\Big(P_M(x)-\lambda\Big(\sum_{i=1}^{n}x_i^2\Big)^2\Big)\in \Sigma_{2r}.$$
Since this is a homogeneous polynomial in $x^{\circ 2}$ we can apply Theorem \ref{theo-squares-vera} to it and conclude that 
\begin{align}\label{aux}
\Big(\sum_{i=1}^{n}x_i\Big)^{2r-2}\Big(p_M(x)-\lambda\Big(\sum_{i=1}^{n}x_i\Big)^2\Big)= 
\sum_{\substack{R\subseteq[n] \\
    |R|\leq 2r, |R|\equiv 2r\ (\text{\rm mod }  2)}}\sigma_R(x)\prod_{i\in R}x_i,
\end{align}
where  $\sigma_R\in\Sigma$ has degree at most $2r-|R|$. Notice that
$$\Big(\sum_{i=1}^{n}x_i\Big)^{2r-2}=\Big(1-1+\sum_{i=1}^{n}x_i\Big)^{2r-2}=1+h(x)\Big(1-\sum_{i=1}^{n}x_i\Big),$$
for some $h\in \mathbb{R}[x]_{2r-3}$. 
 Using this observation,   (\ref{aux}) implies 
$$p_M(x)-\lambda\Big(\sum_{i=1}^{n}x_i\Big)^2\in \mathcal{T}(x_1,\dots, x_n)_r + \Big\langle1-\sum_{i=1}^{n}x_i\Big\rangle_{2r}.$$
Using again (\ref{eq0}) \MoL{(replacing $P_M$ by $p_M$)} we obtain 
$$p_M(x)-\lambda\in \mathcal{T}(x_1,\dots, x_n)_r + \Big\langle1-\sum_{i=1}^{n}x_i\Big\rangle_{2r},$$
which shows $p^{(r)}_{M,po}\ge P^{(2r)}_M$.
\end{proof}

As a direct consequence of Lemmas \ref{sim-sph}, \ref{equiv-sos-lass} and \ref{link-preo_simp} we obtain the following links among the above parameters.

\begin{corollary}\label{cor-link-general}
For any  $M\in\MS^n$ and $r\ge 0$ we have:  
\begin{align}
    p_{\min}\geq P_M^{(2r+2)}=\Theta_M^{(2r)}=p_{M,po}^{(r+1)}\geq p_M^{(r+1)}.
\end{align}
\end{corollary}

\begin{remark}
In view of the formulation (\ref{programp}) for the parameter $p_M^{(1)}$,  the difference with the parameter $p^{(1)}_{M,po}=P^{(2)}_M=\Theta^{(0)}_M$ lies in the fact that, while for $p^{(1)}_M$ we search for a decomposition  $M=\lambda J+ Q+ (ea^T+ae^T)/2\succeq 0$ with $Q\succeq 0$ and $a\in\oR^n_+$, in the definition of $\Theta^{(0)}_M$ we search for a decomposition  $M=\lambda J+Q+N\succeq 0$ with $Q\succeq 0$, but now $N$ can be an arbitrary entry-wise nonnegative matrix.
\end{remark}

\subsection{Application to the stable set problem}

Here we apply   the results in Section \ref{secSQP}  to the formulation of the stability number $\alpha(G)$  via the Motzkin-Straus formulation (\ref{motzkin-form}), the special instance of standard quadratic program where we select the matrix $M=I+A_G$. As in the introduction we set $f_G=p_M$, $F_G=P_M$ and $f_{G,po}=p_{M,po}$ for this matrix $M=I+A_G$. As a direct application of Corollary \ref{cor-link-general}, we obtain
\begin{equation}\label{eqlink2}
{1\over \alpha (G)} \ge F^{(2r+2)}_G=f^{(r+1)}_{G,po}\ge f^{(r+1)}_{G}.
\end{equation}
It remains to  link  the parameters $\vartheta^{(r)}(G)$ and $\Theta^{(r)}_M$ for the matrix $M=I+A_G$.

\begin{lemma}\label{lemlink2}
For any graph $G$ and $r\ge 0$,  we have: $\Theta^{(r)}_{I+A_G}=\dfrac{1}{\vartheta^{(r)}(G)}.$
\end{lemma}

\begin{proof}
Directly from the definitions of $\vartheta^{(r)}(G)$  in (\ref{eqthetar}) and $\Theta^{(r)}_{I+A_G}$ in (\ref{relax-copositive}).
\end{proof}

Combining (\ref{eqlink2}) and Lemma \ref{lemlink2}
we obtain the inequalities claimed in (\ref{eqlink1}), which we repeat here for convenience.

\begin{corollary}\label{cor-link-alpha}
For any graph $G$ and $r\ge 0$ we have
$$\dfrac{1}{\alpha(G)}\geq \dfrac{1}{\vartheta^{(2r)}}=\lsph_G^{(2r+2)}=\lsim_{G,po}^{(r+1)}\geq \lsim_G^{(r+1)}.$$
\end{corollary}

We now use the result of Lemma \ref{ch-qp-level-1} to characterize when the parameter $f^{(1)}_G$ is feasible (and thus exact). 
\begin{lemma}\label{lemexactfG1}
For any graph $G$, the parameter $f^{(1)}_G$ is finite  or, equivalently,  $f^{(1)}_G= 1/\alpha(G)$, if and only if $G$ is a disjoint union of cliques.
\end{lemma}
\begin{proof}
We use Lemma \ref{ch-qp-level-1} applied to the matrix $M=I+A_G$.  First, assume $M=\lambda J+Q+(ae^T+ea^T)/2$ for some $\lambda\in\oR$, $Q\succeq 0$ and $a\in\oR^n_+$,
we show that $G$ is a disjoint union of cliques. For this it suffices to show that $\{1,2\},\{1,3\}\in E$ implies $\{2,3\}\in E$. This follows easily using Lemma \ref{lem-finite-pM1}. Indeed, if $\{1,2\},\{1,3\}\in E$ then we have $M_{11}+M_{22}-2M_{12}=0$ and thus $Q(e_1-e_2)=0$ and, in the same way, $Q(e_1-e_3)=0$. This implies $Q(e_2-e_3)=0$ and thus $M_{22}+M_{33}-2M_{23}=0$, i.e., $\{2,3\}\in E$.
\\
Conversely, assume $G$ is a disjoint union of cliques, say $V=C_1\cup\ldots \cup C_k$ where $k=\alpha(G)$ and each $C_i$ is a clique of $G$. We show that $p^{(1)}_M=p_{\min}$.  For this note that, for any $x\in \Delta_n$, we have 
$$x^T(I+A_G)x= \sum_{i=1}^k \Big(\sum_{j\in C_i}x_j\Big)^2 \ge {1\over k}={1\over \alpha(G)},
$$
where we use Cauchy-Schwartz inequality \MoL{combined with $\sum_{i=1}^k\big( \sum_{j\in C_i}x_j\big) =1$ to derive the inner inequality}. This shows $p^{(1)}_M\ge p_{\min}$ and thus equality holds.
\end{proof}

In Section \ref{sec-convergence-w} we will investigate  finite convergence of the simplex-based Lasserre hierarchy $f^{(r)}_G$, which, in view of 
Corollary \ref{cor-link-alpha}, 
directly implies finite convergence of the  hierarchy $\vartheta^{(r)}(G)$.  
 For this we will use Theorem \ref{theo-Nie}  that requires  to understand the structure of the global minimizers of problem (\ref{motzkin-form}), which is what we do in the next section, \MoL{in the general setting of the weighted stable set problem. }

\section{\MoL{Minimizers of the (weighted) Motzkin-Straus formulation}}\label{minimizers-w}

In this section we prove some properties  of the 
minimizers of the  Motzkin-Straus formulation, in the general setting 
of  the weighted stable set problem. 
We consider a graph $G=(V=[n],E)$ equipped with positive node weights $w\in\oR^V$, i.e.,  with $w_i>0$ for $i\in V$. 
A stable set $S\subseteq V $ is said to be  {\em $w$-maximum} if it  maximizes the function $w(S)=\sum_{i\in S} w_i$ over all  stable sets of $G$ and $\alpha(G,w)$ denotes the maximum weight of a stable set  in $G$. 
We say that an edge $\{i,j\}\in E$ is $w$-\textit{critical} in $G$ if there exists a set $R\subseteq V$ such that both  sets $R\cup\{i\}$ and $R\cup\{j\}$ are $w$-maximum stable sets; note  this implies $\alpha(G,w)=w(R)+w_i=w(R)+w_j$ and thus equality $w_i=w_j$. When $w=e=(1,\ldots,1)$ is the all-ones weight vector the $w$-maximum stable sets are  the maximum stable sets, $\alpha(G,e)=\alpha(G)$  and the $w$-critical edges are the critical edges of $G$. 

Following \cite{GHPR} let us define the matrix $B_w\in \mathcal S^n$, with entries
\begin{equation}\label{eqBw}
(B_w)_{ii}={1\over w_i} \ (i\in [n]),\ (B_w)_{ij}={1\over 2}\big({1\over w_i}+{1\over w_j}\big)\ (\{i,j\}\in E),\ (B_w)_{ij}=0\ (\{i,j\}\in \overline E)
\end{equation}
and the  matrix spaces
\begin{equation}\label{eqNG}
\mathcal N(G)=\{A\in \mathcal S^n: A_{ii}=0\ (i\in [n]),\ A_{ij}\ge 0 \ (\{i,j\}\in E),\ A_{ij}=0\ (\{i,j\}\in \overline E)\},
\end{equation}
\begin{equation}\label{eqmatB}
\mathcal M(G,w)=B_w+\mathcal N(G)=\{B_w+A: A\in  \mathcal N(G)\},
\end{equation}
so that  
\begin{equation}\label{eqmatBapex}
\mathcal M(G,w)=\{B\in \mathcal S^n: B_{ii}={1\over w_i}\ (i\in V),\ B_{ij}\ge {1\over 2}(B_{ii}+B_{jj})\ (\{i,j\}\in E),\ B_{ij}=0 \ (\{i,j\}\in \overline E)\}.
\end{equation}
For the all-ones node weights $w=e=(1,1,\dots, 1)$, we have $B_w=I+A_G$.
We will also need the set 
\begin{equation}\label{setBw}
 \mathcal M^*(G,w) =\{B\in \mathcal M(G,w): 2B_{ij}>B_{ii}+B_{jj}  \text{ for all  } w\text{-critical edges } \{i,j\} \in E\}.
 \end{equation}
Clearly $\mathcal M^*(G,w)$ contains all matrices lying in the relative interior of $\mathcal M(G,w)$ and $\mathcal M^*(G,w)=\mathcal M(G,w)$ if there is no $w$-critical edge in $G$. 

In  \cite{GHPR} it  is shown that, for any matrix $B\in \mathcal M(G,w)$, the weighted stable set problem can be reformulated via the following weighted analogue of the Motzkin-Straus formulation
\begin{equation}\label{motzkin-w}\tag{M-S-weighted}
	\frac{1}{\alpha(G,w)}=\min \{p_B(x)=x^TBx: x\in \Delta_n\}.
\end{equation}
We now  investigate the minimizers of problem (\ref{motzkin-w}), whose structure depends on the weighted graph $(G,w)$ and  on the choice of the matrix $B$ in the set $\mathcal M(G,w)$.
In particular we will show that  their number is finite precisely when $B$ belongs to the set $\mathcal M^*(G,w)$.
 As mentioned earlier  the property of having finitely many minimizers is indeed important in the analysis of the finite convergence of the corresponding Lasserre hierarchy.

We start with a useful property of  local minimizers for a  class of standard quadratic programs. The proof is essentially along the lines of the proof of \cite[Theorem 5]{GHPR} (and is the key argument for showing the equality in (\ref{motzkin-w})).

\begin{lemma}\label{lem-reduce-edge}
	Consider  the standard quadratic program
	\begin{equation}\label{eqpmin}
	p_{\min}=\min \{p_M(x)=x^TMx: x\in \Delta_n\},
	\end{equation}
	where $M$ is a matrix of the form  \begin{equation}\label{eqM}
	M=\left(\begin{matrix} a_1 & b & c_1^T\cr 
	b & a_2 & c_2^T\cr
	c_1 & c_2 & M_0
	\end{matrix}\right),
\end{equation}
with $a_1,a_2 >0$, $b\in \oR$ satisfying  $2b\geq a_1+a_2$, $c_1,c_2\in\oR^{n-2}$ and $M_0\in \mathcal S^{n-2}$. Assume $x$ is a local minimizer of problem (\ref{eqpmin}) with $x_1, x_2 > 0$ and define the vectors  $\tilde{x}=x+x_2(e_1-e_2)$ and $\overline{x}=x-x_1(e_1 -e_2) \in \Delta_n$.
Then, $2b=a_1+a_2$ holds and, for any scalar $\lambda\in [0,1]$, we have $p_M(\lambda \tilde x+(1-\lambda) \overline x)=p_M(x)$.
\end{lemma}

\begin{proof}
Consider the problem 
\begin{equation*}\label{p-aux}
\min_{t\in[-x_2,x_1]} p_M(x_1-t, x_2+t, x_3, \dots, x_n),
\end{equation*}
which can be rewritten as 
\begin{equation}\label{p-aux-2}
 \min_{t\in[-x_2,x_1]} t^2(a_1+a_2 - 2b) + \beta t +\gamma, 
 \end{equation}
 where $\beta, \gamma$ are scalars depending on $M$. By assumption, $t=0$ lies in the interior of the interval $[-x_2,x_1]$ and it is a local minimizer 
of problem (\ref{p-aux-2}). 
 If $a_1+a_2-2b<0$ then the objective function of (\ref{p-aux-2}) is strictly concave and thus it cannot have a local minimum at an interior point of $[-x_2,x_1]$.
Hence $a_1+a_2=2b$ holds. If $\beta\ne 0$ then the objective function  is  linear and thus it again 
cannot have a local minimum in the interior of $[-x_2,x_1]$. Hence we must have $\beta=0$, so that $p_M(x)= p_M(x_1-t, x_2+t, x_3, \dots, x_n)$ for any $t\in [-x_2,x_1]$ or, equivalently,
$p_M(\lambda \tilde x+(1-\lambda)\overline x)=p_M(x)$ for any $\lambda\in [0,1]$.
\end{proof}

We recall a result of \cite{GHPR} that characterizes the global minimizers of ((\ref{motzkin-w}) whose support is a stable set.

\begin{lemma}[\cite{GHPR}]\label{lem_stable_then}
Assume $B\in \mathcal M(G,w)$. Let $x\in\Delta_n$ and assume its support $S=\supp(x)$ is a stable set of $G$. If $x$ is a global minimizer of problem (\ref{motzkin-w}) then $S$ is a $w$-maximum stable set,  $x_i=\frac{w_i}{\alpha(w,G)}$ for $i\in S$ and $x_i=0$ for $i\in V\setminus  S$.
\end{lemma}

\begin{proof}
The argument is classical and based on Cauchy-Schwartz inequality. We have 
$$1=\sum_{i\in S}x_i=\sum_{i\in S} {x_i\over \sqrt{w_i}}\sqrt{w_i}\le \sqrt {\sum_{i\in S} {x_i^2\over w_i}} \sqrt{\sum_{i\in S}w_i}= \sqrt{x^TBx}\sqrt{w(S)}\le \sqrt{x^TBx} \sqrt{\alpha(G,w)},$$
where the last two (in)equalities hold since $S$ is a stable set. By assumption, $x^TBx=1/\alpha(G,w)$ since $x$ is a global minimizer of ((\ref{motzkin-w}). Hence equality holds throughout. Then equality in the first (Cauchy-Schwartz) inequality implies the desired result.
\end{proof}

We now characterize  the global minimizers of problem (\ref{motzkin-w}).

\begin{proposition}\label{global_min}
Assume $B\in \mathcal M(G,w)$. Let $x\in \Delta_n$ with support $S=\supp(x)$ and let $C_1,\ldots,C_k$ denote the connected components of the graph $G[S]$. Then $x$ is a global minimizer of problem (\ref{motzkin-w}) if and only if the following conditions hold:
\begin{description}
\item[(i)] $w_i=w_j$ for all $i,j\in C_h$ and $h\in [k]$,
\item[(ii)] $C_h$ is a clique of $G$ for all $h\in[k]$,
\item[(iii)] $\sum_{i\in C_h} x_i=\frac{w_{i_h}}{\alpha(G,w)}$, where $i_h$ is any given node in $C_h$, for all $h\in [k]$,
\item [(iv)] $2B_{ij}=B_{ii}+B_{jj}=\frac{1}{w_i}+\frac{1}{w_j}$ for all edges $\{i,j\}$ of $G[S]$.
 \end{description}
In that case all the edges of $G[S]$ are $w$-critical. 
\end{proposition}

\begin{proof}
We first show the `if part'. Assume (i)-(iv) hold, we show $x^TBx=1/\alpha(G,w)$ holds.
Using (i)-(iv) we obtain $${1\over \alpha(G,w)}\le x^TBx= \sum_{h=1}^k {1\over w_{i_h}} \big(\sum_{i\in C_h}x_i\big)^2 =\sum_{h=1}^k {1\over w_{i_h}} \big({w_{i_h}\over \alpha(G,w)}\big)^2 
={1\over \alpha(G,w)^2} \sum_{h=1}^k w_{i_h}\le {1\over \alpha(G,w)},
$$
since the set $\{i_h: h\in [k]\}$ is a stable set in $G$. Hence equality holds throughout, which shows the desired result.

We now show the `only if' part. 
 Assume  $x$ is a global minimizer, we show that (i)-(iv) hold. Condition (iv) follows directly using Lemma \ref{lem_stable_then} applied to the matrix $B$. 
Consider nodes  $i_1\in C_1,  \dots, i_k\in C_k$ lying in the different connected components of $G[S]$. Then  $I=\{i_1,\ldots,i_k\}$ is a stable set of $G$. Define the vector $y\in \Delta_n$, with entries 
	$y_{i_h}=\sum_{i\in C_h}x_i$ for $h\in [k]$ and $y_i=0$ for all remaining vertices $i\in V\setminus I$. 
	By applying iteratively Lemma~\ref{lem-reduce-edge} (with the matrix $B$, using the edges in a spanning tree in each connected component $C_h$) we obtain that  $y^TBy=x^TBx$. Hence, $y$ is a global minimizer of (\ref{motzkin-w}) whose support is a stable set and thus,   by Lemma \ref{lem_stable_then}, we obtain that $I$ is a $w$-maximum stable set and $\sum_{i\in C_h}x_i=y_{i_h}=w_{i_h}/ w(I)$ for all $h\in [k]$, so that (iii) holds. 
	Next we check (ii), i.e., that   each component (say) $C_1$ is a clique. Indeed, if $i\ne j\in C_1$ are not adjacent then the set $\{i,j\}\cup\{i_2,\ldots,i_k\}$ is stable and $w( \{i,j\}\cup\{i_2,\ldots,i_k\})> w(\{i,i_2,\ldots,i_k\})=\alpha(G,w)$.  Moreover, the edge $\{i,j\}$ is $w$-critical since both sets $\{i,i_2, \ldots, i_k\}$ and $ \{j,i_2,  \ldots, i_k\}$ are $w$-maximum stable sets. Thus (i) holds and the proof is complete.
\end{proof}

As a direct application we obtain the characterization of the global minimizers of the (unweighted) Motzkin-Straus problem (\ref{motzkin-form}).
\begin{corollary}\label{cor_global_min}
Let $x\in \Delta_n$ with support $S=\supp(x)$ and let $C_1,\ldots,C_k$ denote the connected components of the graph $G[S]$. Then $x$ is a global minimizer of problem (\ref{motzkin-form}) if and only if the following conditions hold:
\begin{description}
\item[(i)]$k=\alpha(G)$,
\item[(ii)] $C_h$ is a clique for all $h\in[k],$
\item[(iii)] $\sum_{i\in C_h} x_i=1/k$ for all $h\in [k]$.
\end{description}
\end{corollary}

As another application we can  characterize when problem (\ref{motzkin-w}) has finitely many minimizers and in addition we show that in this case the sufficient optimality conditions hold at all minimizers.

\begin{proposition}\label{prop-finiteMSw}
Assume $B\in\mathcal M(G,w)$. The following assertions are equivalent.
\begin{description}
\item[(i)] Problem (\ref{motzkin-w}) has finitely many global minimizers.
 \item[(ii)]
$B_{ij}> {1\over 2}\big({1\over w_i}+{1\over w_j}\big)$ for all edges $\{i,j\}\in E$ that are $w$-critical.
\item[(iii)] The optimality conditions (\ref{eqFOOC}), (\ref{eqSCC}) and (\ref{eqSOSC}) hold at all the global minimizers of (\ref{motzkin-w}).
\end{description}
In that case the global minimizers are  the 
vectors $x\in \Delta_n$ with entries $x_i=w_i/\alpha(G,w)$ for $i\in S$ and $x_i=0$ for $i\in V\setminus S$, where $S$ is a $w$-maximum stable set of $G$.
\end{proposition}

\begin{proof}
 We first show (i) $\Longrightarrow$ (ii). For this assume there exists a $w$-critical edge (say) $\{1,2\}\in E$ such that $B_{12}={1\over 2}(1/w_1+1/w_2)$, we show that the number of minimizers is infinite. 
  As $\{1,2\}$ is $w$-critical there exists $R\subseteq V$ such that both sets $R\cup\{1\}$ and $R\cup\{2\}$ are $w$-maximum stable sets. For any scalar $t\in [0,1]$ consider the point $x\in \Delta_n$ with support $S=R\cup\{1,2\}$ and entries $x_1=tw_1/w(S)$, $x_2=(1-t)w_2/w(S)$ and $ x_i=w_i/w(S)$ for $i\in R$. Then, by  Proposition \ref{global_min}, $x$ is a minimizer for all $t\in [0,1]$.
 
 Since the implication (iii) $\Longrightarrow$ (i) is clear, it now suffices to  show  the  implication (ii) $\Longrightarrow$ (iii). 
So assume (ii) holds and consider a global minimizer $u$ of (\ref{motzkin-w}). In view of Proposition \ref{global_min} there exists a $w$-maximum stable set $S$ such that $u_i=w_i/\alpha(G,w)$ for $i\in S$ and $u_i=0$ for $i\in V\setminus S$.
Consider the polynomials  $g_i(x)=x_i$ for $i\in [n]$, $h(x)=\sum_{i=1}^{n}x_i-1$, so that the feasible region of problem (\ref{motzkin-w}) is defined by the constraints $g_i(x)\ge 0$ for $i\in [n]$, and $h(x)=0$.
	The active constraints at $u$ are the constraints $g_i(x)\geq 0$ for $i\notin S$, and $h(x)=0$. Hence $J(u)=V\setminus S$. Clearly, (CQC) holds at $u$ since the gradients of the active constraints at $u$ are the vectors $e$ and $e_i$ for $i\in V\setminus S$, which  are linearly independent (as $S\ne \emptyset $).  Next note that 
	\begin{equation}\label{eqgradfG}
	{\partial p_B\over \partial x_i}(u) =\begin{cases} 
	\frac{2}{\alpha(G,w)} & \text{ if } i\in S, \\
	\frac{2}{\alpha(G,w)}\sum_{j\in N_S(i)} B_{ij} w_j & \text{ if } i \notin S.
	\end{cases}
	\end{equation}
	The first order optimality condition reads 
	$$\nabla p_B(u)=\lambda e + \sum_{i\notin S}\mu_i e_i, \quad
	\text{ where } \lambda \in \oR,\ \mu_i\ge 0 \text{ for } i \in V\setminus S.$$
	Looking at coordinate $i\in S$ we obtain $\lambda=2/\alpha(G,w)$.
	Then we obtain 
	$$\mu_i={2\over\alpha(G,w)} \Big(-1+ \sum_{j\in N_S(i)} w_jB_{ij}\Big)\quad \text{ for each } i\in V\setminus S.$$
Let $i\in V\setminus S$, we show that $\mu_i>0$. For this note that $w_jB_{ij}\ge {w_j\over 2w_i}+{1\over 2}>{1\over 2}$ for all $j\in N_S(i)$. Hence we have $\mu_i>0$ if $|N_S(i)|\ge 2$. So assume now $|N_S(i)|=1$, say $N_S(i)=\{j\}$ so that $\mu_i={2\over \alpha(G,w)}(-1+w_jB_{ij})\ge {1\over \alpha(G,w)}({w_j\over w_i}-1)$.
As $S$ is a $w$-maximum stable set and the set $S\setminus\{j\}\cup \{i\}$ is stable we have $w(S\setminus\{j\}\cup\{i\})\le w(S)$ and thus $w_i\le w_j$.
If $w_j>w_i$ then we have $\mu_i>0$ as desired. So assume now $w_i=w_j$, which implies that the edge $\{i,j\}$ is $w$-critical.
Then, by assumption (ii), we must have   $w_jB_{ij}>{w_j\over 2w_i}+{1\over 2}=1$, which again implies $\mu_i>0$.
So we have shown that  (\ref{eqFOOC}) and (\ref{eqSCC}) hold at $u$.
	Finally, we check that also (\ref{eqSOSC}) holds.  For this consider the Lagrangian function  
	$$L(x)=p_B(x)-\lambda\Big(\sum_{i=1}^n x_i-1\Big)-\sum_{i\in J(u)}\mu_ix_i.$$ 
	Then we have $\nabla^2L(u)=\nabla^2p_B(u)=2B$. Consider a vector $0\neq v\in G(u)^\perp$. Then $v_i=0$ for $i\notin S$, therefore $v^T\nabla^2L(u)v=2\sum_{i\in S} v_i^2>0$ since $v\ne 0$. So (\ref{eqSOSC}) holds at $u$. This concludes the proof.
\end{proof}

Hence, problem (\ref{motzkin-w}) has finitely many minimizers if and only if we choose the matrix $B$  in the set $\mathcal M^*(G,w)$ as defined in (\ref{setBw}).
This is the case, for example, when  $B$ lies in the relative interior of $\mathcal M(G,w)$ as observed in  \cite{GHPR}. Clearly $\mathcal M^*(G,w)=\mathcal M(G,w)$ if there is no $w$-critical edge in $G$. In the unweighted case, one can for instance select $B=I+2A_G\in \mathcal M^*(G,e)$ as perturbation of the adjacency matrix, as already observed earlier, e.g., in \MoL{\cite{Bomze1997,PJ1996}. Recent work, e.g., in  \cite{BRZ2021,HR2019} uses such perturbed (aka regularized) formulations to approximate the maximum stable problem by applying first-order methods. }

\section{\MoL{ Finite convergence and perturbed hierarchies}}\label{sec-convergence-w}

In this section we give a partial positive answer to Conjecture \ref{conj2} and show that the de Klerk-Pasechnik hierarchy $\vartheta^{(r)}(\cdot)$ has finite convergence for the class of acritical graphs. Our approach relies on proving finite convergence for acritical graphs of the (weaker) simplex-based Lasserre hierarchy  $f^{(r)}_G$ in (\ref{eqlassimplex}) corresponding to problem (\ref{motzkin-form}).
We in fact  show a more general  result in the setting of weighted graphs.

\subsection{\MoL{Finite convergence of the Lasserre hierarchy   for the (weighted) Motzkin-Straus formulation}}

In  Section \ref{minimizers-w} we  showed that, if in  the (weighted) Motzkin-Straus problem (\ref{motzkin-w})  we choose the matrix $B$ to lie in the set $\mathcal M^*(G,w)$ from (\ref{setBw}), then there are finitely many minimizers  and the  suffiicient optimality conditions hold at all of them (Proposition \ref{prop-finiteMSw}).  Hence we can then apply Theorem \ref{theo-Nie} and conclude the finite convergence of the corresponding simplex-based Lasserre hierarchy $p_B^{(r)}$ in (\ref{las-simplex-general}) and thus also of the bounds $\Theta_B^{(r)}$ in (\ref{relax-copositive}).

\begin{theorem}\label{theo-finite-w}
Let $(G,w)$ be a weighted graph with positive node weights $w>0$.
Consider problem (\ref{motzkin-w}) where the matrix $B$ belongs to $\mathcal M^*(G,w)$. Then the following holds.
\begin{description}
\item[(i)] $p_B^{(r)}={1\over \alpha(G,w)}$ for some $r\in \oN$.
\item[(ii)] $\Theta_B^{(r)}={1\over \alpha(G,w)}$ for some $r\in \oN$.
\end{description}
In particular, if $G$ has no $w$-critical edge then (i), (ii) hold for any matrix $B\in \mathcal M(G,w)$ and thus for the matrix $B_w$.
\end{theorem}
\begin{proof}
(i) follows  using Theorem \ref{theo-Nie} and Proposition \ref{prop-finiteMSw}). Then (ii) follows from (i) in view of Corollary \ref{cor-link-general}.
\end{proof}

\begin{corollary}\label{finite-acritical-w}
Assume $G$ is a graph with no critical edges. Then the following holds.
\begin{description}
\item[(i)] $f^{(r)}_G = {1\over \alpha(G)}$ for some $r\in \oN$.
\item[(ii)] $\vartheta^{(r)}(G)=\alpha(G)$ for some $r\in \oN$.
\end{description}
Therefore Conjecture \ref{conj2} holds for the class of acritical graphs.
\end{corollary}

\begin{proof}
This is a direct consequence of Theorem \ref{theo-finite-w}, applied to the all-ones node weights $w=e$ and the matrix $B=I+A_G$, in which case we have $f^{(r)}_G= p_B^{(r)}$ and $\vartheta^{(r)}(G)={1\over \Theta_B^{(r)}}$.
\end{proof}


\ignore{
\begin{theorem}\label{theo-finite-las-acritical}
Let $G$ be a graph with no critical edges. Then, $\lsim_{G}^{(r)}=1/\alpha(G)$ for some $r\in \mathbb{N}$.
\end{theorem}
}
\ignore{
\begin{corollary}\label{finite-acritical}
Let $G$ be a graph with no critical edges. Then, $\vartheta^{(r)}(G)=\alpha(G)$ for some $r\in \mathbb{N}$.
\end{corollary}
\begin{proof}
This follows directly from Theorem \ref{theo-finite-las-acritical} and Corollary \ref{cor-link-alpha}.
\end{proof}
}

Hence, for problem (\ref{motzkin-w}), having finitely many global minimizers implies finite convergence of the Lasserre hierarchy. Note, however, that this fact does not extend to general polynomial optimization problems. We illustrate this through two examples, first on an instance of standard quadratic program and second on a well-known instance of polynomial optimization problem over the ball.

\begin{example}
Consider the standard quadratic problem (\ref{quadr-simplex}), asking to find the minimum value $p_{\min}$ of the quadratic $p_M(x)=x^TMx$ over the simplex $\Delta_n$, and the corresponding simplex-based Lasserre bounds $p^{(r)}_M$ in (\ref{las-simplex-general}) and copositive-based bounds $\Theta_M^{(r)}$ in (\ref{relax-copositive}).
In \cite{LV2021} we construct a family of copositive  matrices $M$ with size $n\ge 6$ that satisfy $x^TMx=0$ for a unique vector $x\in \Delta_n$ and with  the additional property that $M$ does not belong to any of the cones $\mathcal K^{(r)}_n$. 
Then we have $p_{\min}=0$ and the polynomial $x^TMx$  has a unique minimizer in $\Delta_n$. On the other hand, we must have $\Theta^{(r)}_M<0$ for all $r\in \oN$, because equality  $\Theta^{(r)}_M=0$ would imply $M\in \MK^{(r)}_n$ (since the optimum is attained in  the program (\ref{relax-copositive})).
By Corollary \ref{cor-link-general} we have $p_{\min}\ge \Theta^{(2r)}_M\ge p_M^{(r+1)}$ for all $r\in \oN$. Hence  strict inequality $p_M^{(r)}<p_{\min}$ holds for all $r \in\oN$ and thus  the Lasserre hierarchy does not have finite convergence.
\end{example}

\begin{example} (See, e.g.,  \cite[Example 6.19]{Laurent2009}). Consider the problem of minimizing a polynomial $p$ over the unit ball in $\oR^n$. Assume  $p$ is  homogeneous, $p(x)>0$ for all $x\in \oR^n\setminus \{0\}$, and $p$ is not a sum of squares of polynomials.  Then the minimum of $p$ over the unit ball  is $p_{\min}=0$ and the origin  is the unique global minimizer. However it is known that  the corresponding Lasserre hierarchy does not have finite convergence, see Example 6.19 in \cite{Laurent2009} for details. The main reason is that a decomposition of the form
$p=s_0+s_1(1-\sum_ix_i^2)$ with $s_0,s_1\in \Sigma$ would imply  $p\in \Sigma$. 
For the polynomial $p$ one may, for instance, consider a perturbation of the Motzkin form:
$p_\epsilon= x_1^4x_2^2+x_1^2x_2^4-3x_1^2x_2^2x_3^2+x_3^6+\epsilon(x_1^6+x_2^6+x_3^6)$, selecting $\epsilon>0$  such that $p_\epsilon\not\in\Sigma$.
\end{example}

\subsection{\MoL{Copositive-based bounds for the (weighted) Motzkin-Straus formulation}}

Let $(G,w)$ be a weighted graph with positive node weights $w>0$. 
As a direct consequence of the weighted Motzkin-Straus formulation (\ref{motzkin-w}), for any matrix $B\in \mathcal M(G,w)$  we obtain the following copositive  programming formulation
\begin{align}\label{copositive_epsilon}
    \alpha(G,w)=\min\{ t:  tB-J \in \COP_n\}
\end{align}
for the weighted stability number.
Let us write $B=B_w+A$, where $A$ lies in the set $\mathcal N(G)$ from (\ref{eqNG}).
 In analogy to (\ref{eqzetar}) and (\ref{eqthetar}) we can define    the associated linear and semidefinite  bounds 
\begin{align}
 \zeta_{A}^{(r)}(G,w)=\min \{t: t(B_w+A)-J \in \Cr_n\},\label{liner_deK_epsilon}\\
 \vartheta_{A}^{(r)}(G,w)=\min\{ t:  t(B_w+A)-J \in \kr_n\},  \label{sdp_deK_epsilon}
 \end{align}
that satisfy $\alpha(G,w)\le \vartheta^{(r)}_A(G,w)\le \zeta^{(r)}_A(G,w)$ for all $A\in \mathcal N(G)$.
For the zero matrix $A=0$ we may omit the index and simply write $\zeta^{(r)}_0(G,w)=\zeta^{(r)}(G,w)$ and $\vartheta^{(r)}_0(G,w)=\zeta^{(r)}(G,w)$.
In addition, in the unweighted case when $w=e$,   we have 
$\zeta^{(r)}(G,e)=\zeta^{(r)}(G)$ and $\vartheta^{(r)}(G,e)=\vartheta^{(r)}(G)$.
Note also that for any matrix $A$ we have  $\vartheta^{(r)}_A(G,w)={1\over \Theta^{(r)}_{B_w+A}}$, where $\Theta^{(r)}_{B_w+A}$ is as defined in (\ref{relax-copositive}).

From the previous section we know that the hierarchy $\vartheta^{(r)}_A(G,w)$ converges in finitely many steps to $\alpha(G,w)$ when the matrix $B_w+A$ belongs to the set $\mathcal M^*(G,w)$. In general one may ask whether this holds for any choice of $A\in \mathcal N(G)$, 
which would be the weighted analogue of Conjecture \ref{conj2}.
In fact it would suffice to show this for  the case $A=0$, 
which follows from the monotonicity properties of the bounds with respect to the choice of $A$,  shown in the next lemma.

\begin{lemma}\label{lem-mono}
Let $A_1,A_2\in \mathcal N(G)$. If $A_1\ge A_2$ then we have $\zeta^{(r)}_{A_1}(G,w)\le \zeta^{(r)}_{A_2}(G,w)$ and 
$\vartheta^{(r)}_{A_1}(G,w)\le \vartheta^{(r)}_{A_2}(G,w)$ for all $r\in \oN$. In particular we have $\zeta^{(r)}_A(G,w)\le \zeta^{(r)}(G,w)$ and $\vartheta^{(r)}_A\le \vartheta^{(r)}(G,w)$ for all $A\in \mathcal N(G)$.
\end{lemma}

\begin{proof}
Assume $t$ is feasible for $\zeta^{(r)}_{A_2}(G,w)$, i.e., $t(B_w+A_2)-J\in \MC^{(r)}_n$. Then, 
$t(B_w+A_1)-J= t(B_w+A_2)-J + t(A_1-A_2)\in \MC^{(r)}_n$ since the matrix $t(A_1-A_2)$ is entrywise nonnegative and thus belongs to $\MC^{(r)}_n$.
Hence $t$ is feasible for $\zeta^{(r)}_{A_1}(G,w)$, which shows $\zeta^{(r)}_{A_1}(G,w)\le \zeta^{(r)}_{A_2}(G,w)$. The same argument shows $\vartheta^{(r)}_{A_1}(G,w)\le \vartheta^{(r)}_{A_2}(G,w)$ and the last claim  follows since $A\ge 0$ for any $A\in\mathcal N(G)$.
\end{proof}

As we now show, the linear bounds $\zeta^{(r)}_A(G,w)$ in fact do {\em not} depend on the specific choice of the matrix $A$ in $\mathcal N(G)$.


\begin{theorem}\label{theo-linearA}
For all $r\in \mathbb{N}$ and $A\in \mathcal N(G)$, we have $\zeta_{A}^{(r)}(G,w)=\zeta^{(r)}(G,w)$.
\end{theorem}
\begin{proof}
We only need to show the  inequality $\zeta^{(r)}_{A} (G,w)\ge \zeta^{(r)}(G,w)$.
 For this assume the matrix $t(B_w+A)-J$ belongs to the cone $\MC^{(r)}_n$, we show that also the matrix $tB_w-J$ belongs to $\MC^{(r)}_n$, which implies the desired inequality.
For short, set $B=B_w+A$. By assumption, $tB-J\in \MC^{(r)}_n$, which means that the polynomial 
$(\sum_i x_i)^r x^T(tB-J) x$ has nonnegative coefficients. 
Following \cite{Bomze}, for any matrix $M$  and $r\in\oN$, we have
$$
\Big(\sum_i x_i\Big)^rx^TMx=\sum_{\beta\in I(n,r+2)}\frac{r!}{\beta!}c_\beta x^{2\beta},\quad 
\text{ with } \ c_\beta:=\beta^TM\beta-\beta^T\text{diag}(M),
$$
where 
 $\text{diag}(M)\in \oR^n$ is the vector $(M_{ii})_{i=1}^n$ consisting of the diagonal entries  of $M$.
Hence the polynomial $(\sum_i x_i)^r x^TMx$ has nonnegative coefficients if and only if $c_\beta \geq 0 $ for all $\beta\in I(n,r+2)$. We will now prove that, for the matrix $M=tB-J=t(B_w+A)-J$, the property of having $c_\beta\geq 0$ for all $\beta\in I(n,r+2)$ is in fact independent on the choice of $A\in \mathcal N(G)$.
For this let $\beta\in I(n,r+2)$. Using the fact that $e^T\beta=r+2$ we have
\begin{align*}
    c_\beta&=\beta^T(tB-J)\beta-\beta^T\text{diag}(tB-J) = t(\beta^T B \beta-\beta^T\text{diag}(B_w)) -(r+1)(r+2).
\end{align*}
Therefore, $c_\beta\ge 0$ for all $\beta\in I(n,r+2)$ if and only if  $t \varphi^* \ge (r+1)(r+2)$, where $\varphi^*$ is defined by 
\begin{equation}\label{eqbetaw}
\varphi^*:= \min \{\varphi(\beta) :=\beta^T B \beta-\beta^T\text{diag}(B_w) :  \beta\in I(n,r+2)\}.
\end{equation}
We now show that the optimum value of the program (\ref{eqbetaw}) is attained at some $\beta\in I(n,r+2)$ whose support is a stable set of $G$, using a similar argument as for  Lemma \ref{lem-reduce-edge}. Assume  $\beta^*=(\beta_1^*, \beta_2^*, \dots, \beta_n^*)$ is a minimizer of  problem (\ref{eqbetaw}) with $\beta_1^*, \beta_2^*>0$ for some edge $\{1,2\}\in E$. We show that there exists another minimizer $\beta$ of (\ref{eqbetaw}) of the form 
$\beta=(\beta_1^*+\beta_2^*,0,\beta_3^*,\ldots,\beta_n^*)$ or $(0,\beta_1^*+\beta_2^*,\beta_3^*,\ldots,\beta_n^*)$,
thus with $\beta_1\beta_2=0$. 
For this we consider  problem (\ref{eqbetaw}) restricted to the vectors of the form $(\beta_1^* -\lambda , \beta_2^*+\lambda, \beta_3^*, \dots, \beta_n^*)$  with $\lambda\in \oZ\cap [-\beta_2^*, \beta_1^*]$, which reads
\begin{align}\label{phi-t}
\min_{\lambda\in\oZ\cap [-\beta_2^*, \beta_1^*]} \varphi(\beta_1^* -\lambda , \beta_2^*+\lambda, \beta_3^*, \dots, \beta_n^*).
\end{align}
Observe that the objective value of problem (\ref{phi-t}) takes the form
$$
\varphi(\beta_1^* -\lambda , \beta_2^*+\lambda, \beta_3^*, \dots, \beta_n^*)=\lambda^2(B_{11}+B_{22}-2B_{12}) + c\lambda+d
$$
for some scalars $c,d$, 
and thus it  is concave in $\lambda$. Hence, the minimum value of  (\ref{phi-t})
 is attained at one of the end points of the interval $\oZ\cap [-\beta_2^*, \beta_1^*]$, which shows the desired result.
Repeating this reasoning to any other edge contained in the support of $\beta^*$ we obtain another  minimizer $\beta$ of (\ref{eqbetaw}) whose support is a stable set of $G$.
This shows that  the optimum value of (\ref{eqbetaw})  remains the same when selecting $A=0$. 
Therefore, if the polynomial $(\sum_i x_i)^r x^T(tB-J)x $ has nonnegative coefficients then also the polynomial
$(\sum_i x_i)^r x^T(tB_w-J)x$ has nonnegative coefficients. This concludes the proof.
\end{proof}


In \cite{PVZ2007} it is shown that strict inequality $\alpha(G) < \zeta^{(r)}(G)$ holds for all $r\in \oN$ when $G$ is not a complete graph.
We extend this result to the weighted case and characterize when equality $\zeta^{(r)}(G,w)=\alpha(G,w)$ holds  for some $r\in\oN$.  

\ignore{
\begin{lemma}
\MoL{Consider a graph $(G,w)$ with  positive node weights, ordered (say) as $w_1\ge w_2 \ge \ldots \ge w_n>0$, and define the set $I=\{i\in [n]: w_i=w_1\}$. Assume that equality $\zeta^{(r)}_A(G,w)= \alpha(G,w)$ holds for some $r\in \oN$. Then the following holds. \textcolor{blue}{(The equality $\zeta^{(r)}_A(G,w)=\alpha(G,w)$ holds if and only if (i) holds})
\begin{description}
\item[(i)]  $\alpha(G,w)=w_1$, and thus the $w$-maximum stable sets are the singleton sets $\{i\}$ for $i\in I$.
\item[(ii)] $G[I]$ is a clique and every pair $\{i,j\}$ with $i\in I$ and $j\in V\setminus I$ is an edge.
\item[(iii)] $\alpha(G,w)>\alpha(G[V\setminus I],w).$
\end{description}
In particular, for the all-ones weight $w=e$ we obtain that $G$ is the complete graph (as shown in \cite{PVZ2007}).}
\end{lemma}
}

\begin{lemma}
Consider a graph $(G,w)$ with  positive node weights, ordered (say) as $w_1\ge w_2 \ge \ldots \ge w_n>0$.
 Then, equality $\zeta^{(r)}(G,w)= \alpha(G,w)$ holds for some $r\in \oN$ if and only if $\alpha(G,w)=w_1$.
\end{lemma}
\begin{proof}
Assume  $\zeta^{(r)}(G,w)= \alpha(G,w)$ for some $r\in \oN$. Then  the polynomial $q(x)=(\sum_i x_i)^r  x^T(\alpha(G,w)B_w-J)x^T$ has nonnegative coefficients. Let $S$ be a $w$-maximum stable set and let $u$ be the corresponding minimizer, with entries $u_i=w_i/\alpha(G,w)$ for $i\in S$ and $u_i=0$ otherwise. Then $u$ is a zero of $q(x)$. Pick an index $i\in S$ and consider the coefficient of $q$ for the monomial
$x_i^{r+2}$, which is equal to $-1+\alpha(G,w)/w_i$. Since $q(u)=0$ and $u_i\ne 0$ we must have $\alpha(G,w)=w_i$. This implies that $S=\{i\}$ and thus $w_i=w_1=\alpha(G,w)$.
\\
Assume now $\alpha(G,w)=w_1$; we show $\zeta^{(r)}(G,w)=\alpha(G,w)$, i.e., $M:=\alpha(G,w)B_w-J\in \mathcal C^{(r)}_n$, for some $r\in \oN$.
Note that the set $R=\{i\in V: w_i=w_1\}$ induces a clique in $G$. Then the columns/rows of $M$ indexed by nodes in $R$ are all identical. Since deleting repeated rows/columns preserves membership in  the cone $\mathcal{C}^{(r)}$ we can assume without loss of generality that  $R=\{1\}$.
Hence, $\{1\}$ is the only $w$-maximum stable set and  the polynomial $p_M(x)=x^TMx$ has a unique zero in the simplex, located at the corner $e_1$. Note also   $M_{1j}=(w_1/w_j-1)/2>0$ for all $j\in V\setminus \{1\}$. Hence we may apply Corollary 2 in \cite{CPR} and conclude that 
there exists an $r\in\mathbb{N}$ for which the polynomial $(\sum_{i=1}^{n}x_i)^rx^TMx$ has nonnegative coefficients, so that $M\in \mathcal{C}_n^{(r)}$. \end{proof}


In Theorem \ref{theo-linearA} we saw that the linear hierarchy $\zeta^{(r)}_A(G,w)$ does not depend on the choice of $A\in \mathcal N(G)$. For the 
semidefinite hierarchy  $\vartheta_A^{(r)}(G,w)$ we can  prove this property only for  the first level of the hierarchy. 

\begin{lemma}
For any  $A\in \mathcal N(G)$ and node weights $w>0$ we have $\vartheta_A^{(0)}(G,w)=\vartheta^{(0)}(G,w)$ and thus, in particular,
 $\vartheta^{(0)}_A(G)=\vartheta^{(0)}(G)$.
\end{lemma}

\begin{proof}
We need to show the inequality $\vartheta^{(0)}(G,w)\le \vartheta^{(0)}_A(G,w)$ (since the reverse one is clear). For this 
 let $t$ be feasible for $\vartheta_A^{(0)}(G,w)$, we show that $t$ is also feasible for $\vartheta^{(r)}(G)$. 
Set $B=B_w+A$.  By assumption,  the matrix $tB-J$ belongs to $ \mathcal{K}^{(0)}$, 
i.e., there exists a matrix $P\succeq 0$ such that $\text{diag}(P)=\text{diag}(tB-J)$  and $P\leq tB-J$
(recall the characterization of $\mathcal K^{(0)}$ in Proposition \ref{prop-K0}).
As $\text{diag}(tB-J)=\text{diag}(tB_w-J)$ and both $B$ and $B_w$ have zero entries at positions corresponding to non-edges it suffices to check that, for any edge $\{i,j\}\in E$, 
$P_{ij}\le (tB_w-J)_{ij}$. This follows directly from the fact $2P_{ij}\le P_{ii}+P_{jj} = (tB_w-J)_{ii}+ (tB_w-J)_{jj}= 2 (tB_w-J)_{ij}$, where the first inequality holds since $P\succeq 0$.
\end{proof}

\begin{question}
Given a weighted graph $(G,w)$ with positive node weights $w>0$, is it true that,  for  any $A\in\mathcal N(G)$ and any $r\in \mathbb{N}$, 
we have $\vartheta_A^{(r)}(G,w)=\vartheta^{(r)}(G,w)$? 
\end{question}
Clearly, a positive answer to this question for the all-ones node weights $w=e$ would imply the finite convergence of the hierarchy $\vartheta^{(r)}(G)$ and thus settle Conjecture \ref{conj2}. In fact a positive answer to the following question would also suffice to settle Conjecture \ref{conj2}.

\begin{question}
Given a graph $G$ is  it true that   there exists a matrix $A\in \mathcal N(G)$ such that $I+A\in\mathcal M^*(G,e)$ (i.e., $A_{ij}>0$ for all critical edges $\{i,j\}\in E$) and 
$\vartheta_A^{(r)}(G)=\vartheta^{(r)}(G)$ for all $r\in \oN$? 
\end{question}


\section{Complexity of deciding finiteness of the global minimizers}\label{sec-complexity}

As we saw earlier, having finitely many minimizers is a property which plays an important role in the study of finite convergence of Lasserre hierarchy for polynomial optimization.
This raises the question of understanding the complexity of deciding  whether a polynomial optimization problem has finitely many minimizers. Here, as a byproduct of our results in the previous sections about global minimizers of standard quadratic programs,  we show that unless P=NP there is no polynomial-time algorithm to decide whether a standard quadratic program 
has finitely many global minimizers. The complexity of several other decision problems about \MoL{minimizers} in polynomial optimization has been studied rencently in \cite{AZ2020b,AZ2020a}. In particular,  Ahmadi and Zhang  \cite{AZ2020a}  show that it is strongly NP-hard to decide whether a  polynomial of degree 4 has a local minimizer over $\oR^n$; they also show that the same holds for deciding if    a quadratic polynomial has a local minimizer (or a strict local minimizer) over a polyhedron.
In addition they show  that unless P=NP there cannot be a polynomial-time algorithm that finds a point within Euclidean distance $c^n$ (for any constant $c\ge 0)$ of a local minimizer of an $n$-variate quadratic polynomial over a polytope.

In this section we consider the following problem:\medskip

\noindent
\textbf{FINITE-MIN:} Given an instance of   problem  \eqref{poly-opt}, does it  have finitely many global minimizers? 

\medskip
Consider first the case  when (P) is a linear optimization problem, i.e., when the objective and the constraints  are linear polynomials. Then   the problem is convex and thus, if $x_1$ and $x_2$ are two distinct global minimizers then, for every $0\leq t\leq 1$, the point $z=tx_1+(1-t)x_2$ is also a global minimizer. Hence the problem  has finitely many minimizers if and only if it has a unique one.   Therefore, the problem of deciding whether a linear program has finitely many global minimizers is equivalent to the problem of deciding whether it has a unique optimal solution and a polynomial-time algorithm for this problem was given by Appa  \cite{Appa}.

In the rest of the section we prove that  problem {FINITE-MIN}  is hard
even when restricting to standard quadratic programs.
For this, we first consider the following combinatorial problems,  which we will  use to prove this hardness result.  Recall that given a graph $G=(V,E)$, an edge $e\in E$ is critical if $\alpha(G\setminus e)=\alpha(G)+1$.

\textbf{CRITICAL-EDGE:} Given a graph $G=(V,E)$ and an edge $e\in E$, is $e$ a critical edge of $G$?

\textbf{STABLE-SET:} Given a graph $G$ and $k\in \mathbb{N}$, does $\alpha(G)\geq k$ hold?

The problem STABLE-SET is well-known to be NP-Complete \cite{Karp}. 
From this we now prove that unless P=NP there is no polynomial-time algorithm to decide whether an edge is critical.

\begin{theorem}\label{theo-find-critical}
If there is a polynomial-time algorithm that solves the problem {\rm CRITICAL-EDGE}, then P=NP.
\end{theorem}

\begin{proof}
 Assume that there exists a polynomial-time algorithm for CRITICAL-EDGE; we show how to use it to solve STABLE-SET. 
For this let $G=([n],E)$ be an instance of STABLE-SET
and order its edges  as $e_1,e_2,\dots, e_m$. Then, for each $i=1,2,\dots, m$, we  check whether the edge  $e_i$ is critical in the graph $G_{i-1}:=G\setminus\{e_1,\dots, e_{i-1}\}$. If the answer is yes then 
we have  $\alpha(G_i)=\alpha(G_{i-1})+1$ and, otherwise,  $\alpha(G_i)=\alpha(G_{i-1})$. 
  After checking all the $m$ edges we end up with the empty graph $G_m$ on  $n$ nodes, with  $\alpha(G_m)=n$. Let $p$ be the number of critical edges that have been encountered while checking all the $m$ edges. Then we have
  $n=\alpha(G_m)=p+\alpha(G)$ and thus $\alpha(G)=n-p$ has been computed. Hence a polynomial-time algorithm for CRITICAL-EDGE implies a polynomial-time algorithm for computing $\alpha(G)$.
\end{proof}

Using this reduction we now prove that the problem of deciding whether a standard quadratic optimization problem has finitely many optimal solutions is hard. For this, given a graph $G=([n],E)$ and a fixed  edge $e\in E$, consider the following standard quadratic program:
\begin{align}\label{motzkin-e}
\MoL{{1\over \alpha(G)}}=    \min \ x^T(I+A_G+ A_{G\setminus{e}})x \ \text{ subject to } \ x\geq 0,\  \sum_{i=1}^{n}x_i=1,
\end{align}
where 
 in the matrix defining the objective function, all edges of $G$ get weight 2, except the selected edge $e$ which keeps weight 1. The fact  that the optimum value of (\ref{motzkin-e}) is equal to $1/\alpha(G)$ \MoL{follows since this is an instance of (\ref{motzkin-form-w}) with $B=B_w+A$, where  $w=e$ is the all-ones weight vector, $B_e=I+A_G$ and $A=A_{G\setminus e}$.}


\MoL{As a direct application of Proposition \ref{prop-finiteMSw} we obtain the following result.}

\begin{corollary}\label{theo-critical-e}
Given a graph $G=(V,E)$ and an edge $e\in E$, problem (\ref{motzkin-e}) has infinitely many global minimizers if and only if $e$ is a critical edge of $G$.
\end{corollary}

Combining Theorem \ref{theo-find-critical} and Corollary \ref{theo-critical-e} gives directly the following hardness result.

\begin{corollary}\label{cor-hard}
If there is a polynomial-time algorithm to decide whether a standard quadratic program  has finitely many global minimizers then P=NP.
\end{corollary}

\section{Concluding remarks}

We have shown finite convergence of the de Klerk-Pasechnik hierarchy $\vartheta^{(r)}(G)$ for the class of acritical graphs by relating it to the sum-of-squares hierarchy  (\ref{eqlassimplex})  for the Motzkin-Straus formulation of $\alpha(G)$; moreover we showed this for the general setting of the weighted stable set problem. Proving finite convergence for all graphs remains wide open. In fact, as we have observed, it would be sufficient to show finite convergence   for the class of critical graphs.
The hierarchy (\ref{eqlassimplex}) however is weaker than the sum-of-squares hierarchy (\ref{eqlaspo}) based on using the preordering (generated by the polynomials defining the simplex $\Delta_n$), which we have shown to be equivalent to the hierarchy $\vartheta^{(r)}(G)$. A possible approach to solve Conjecture \ref{conj2} could therefore be to fully exploit this \MoL{additional} real algebraic structure. Another approach could be to use the perturbed sum-of-squares hierarchies that we have introduced and for which we could show finite convergence; such a strategy would require to be able to show degree bounds on the level of finite convergence that do not depend on the perturbation parameter.

Showing the stronger Conjecture \ref{conj1}, which asks whether $\rrank(G)\le \alpha(G)-1$,  seems even more challenging. The resolution in \cite{GL2007} for graphs with small stability number $\alpha(G)\le 8$ required  technically involved arguments. It is likely that the full resolution will need a new set of dedicated tools.  As pointed out  in \cite{GL2007}, one of the difficulties lies in understanding the behaviour of the $\rrank$ 
 under the operation of adding isolated nodes.  We will further investigate this question in the follow-up work \cite{LV2021}. 
 \ignore{
 \textcolor{red}{Given a graph  $G=(V,E)$ and a node $i\in V$, let $G_i$ denote the graph obtained from $G$ by deleting the closed neighbourhood $N(i)\cup\{i\}$ and let $\widehat G_i=G_i\oplus i$ be obtained by adding back node $i$ as an isolated node to $G_i$, so that  $\alpha(G_i)\le \alpha(G)-1$ and $\alpha(\widehat G_i)=\alpha(G_i)+1\le \alpha(G)$. One can verify that 
 \begin{equation}\label{eqinduction}
 \rrank(G)\le 1+ \max\{\rrank(\widehat G_i): i\in V\}
 \end{equation}
  (see \cite{GL2007}). Hence, if adding an isolated node would not increase the rank, then  Conjecture \ref{conj1} would follow  using induction. Indeed one would have $\rrank(\widehat G_i)\le \rrank(G_i)\le \alpha(G_i)-1 \le \alpha (G)-2$, which, using (\ref{eqinduction}), would imply  $\rrank(G)\le \alpha(G)-1$, thus settling Conjecture \ref{conj1}. However,  it is not true in general that adding isolated nodes does not increase the rank. Indeed,  one can show that adding eight isolated nodes to $C_5$ produces a graph $G$ with rank 1, but adding one more isolated node to $G$ gives a graph with rank at least 2  \cite{LV2021}. 
}}

While we could characterize the graphs for which the first level of the sum-of-squares hierarchy (\ref{eqlassimplex}) (at order $r=1$) is exact, the analogous question for the first level of the pre-ordering based hierarchy (\ref{eqlaspo}) is much more difficult. This is equivalent   to 
 understanding which graphs have $\rrank$ 0, a question which we will investigate in \cite{LV2021} and where critical edges also play a crucial role. 

\subsubsection*{Acknowledgements}
\MoL{We thank the two referees for their useful feedback on an earlier version of the paper that helped us improve the presentation. In particular we thank a referee for his/her pointers to the literature and for his/her insightful questions that led us to present some of the results in the general setting of the weighted stable set problem.}

\end{document}